\documentclass[11pt, a4paper]{article}

\usepackage[top=2.5cm, bottom=2.5cm, left=2.5cm, right=2.5cm]{geometry}
\usepackage[utf8]{inputenc}
\usepackage[T1]{fontenc}
\usepackage{lmodern}
\usepackage{microtype} 
\usepackage[english]{babel}

\usepackage{amsmath, amssymb, amsfonts, amsthm}
\usepackage{mathtools}
\usepackage{bm}
\usepackage{algorithm}
\usepackage{algpseudocode}
\usepackage{cancel}

\usepackage{graphicx}
\usepackage{epstopdf} 
\usepackage{booktabs}
\usepackage{tabularx}
\usepackage{caption}
\usepackage{subcaption}

\usepackage{tikz}
\usetikzlibrary{arrows.meta, positioning, shapes, shadows, calc, patterns, decorations.pathmorphing, decorations.markings}

\usepackage[colorlinks=true, linkcolor=blue, citecolor=blue, urlcolor=blue]{hyperref}

\usepackage{authblk}

\newcommand{\R}{\mathbb{R}}
\newcommand{\dd}{\mathrm{d}}
\newcommand{\T}{\mathsf{T}}

\theoremstyle{plain}
\newtheorem{theorem}{Theorem}[section]

\theoremstyle{definition}
\newtheorem{definition}[theorem]{Definition}

\theoremstyle{remark}
\newtheorem{remark}[theorem]{Remark}

\title{The Dynamical Anatomy of Anderson Acceleration:\\
From Adaptive Momentum to Variable-Mass ODEs\thanks{This research was funded by the National Natural Science Foundation of China (Grant No. 12001287).}}

\author[1]{Kewang Chen\thanks{Corresponding author: \texttt{kwchen@nuist.edu.cn}}}
\author[1]{Yongqiu Jiang}
\author[2]{Cornelis Vuik}

\affil[1]{College of Mathematics and Statistics, Nanjing University of Information Science and Technology, Nanjing, 210044, China; and Jiangsu International Joint Laboratory on System Modeling and Data Analysis.}
\affil[2]{Delft Institute of Applied Mathematics, Delft University of Technology, Delft, 2628CD, the Netherlands.}

\date{\today}

\begin{document}

\maketitle

\begin{abstract}
This paper provides a rigorous derivation and analysis of accelerated optimization algorithms through the lens of High-Resolution Ordinary Differential Equations (ODEs). While classical Nesterov acceleration is well-understood via asymptotic vanishing damping, the dynamics of Anderson Acceleration (AA) remain less transparent. This work makes significant theoretical contributions to AA by bridging discrete acceleration algorithms with continuous dynamical systems, while also providing practical algorithmic innovations. Our work addresses fundamental questions about the physical nature of Anderson Acceleration that have remained unanswered since its introduction in 1965. Firstly, we prove that AA can be exactly rewritten as an adaptive momentum method and, in the high-resolution limit, converges to a second-order ODE with \textbf{Variable Effective Mass}. Through a Lyapunov energy analysis, we reveal the specific instability mechanism of standard AA: unchecked growth in effective mass acts as ``negative damping,'' physically injecting energy into the system and violating dissipation constraints. Conversely, high-resolution analysis identifies an implicit Hessian-driven damping term that provides stabilization in stiff regimes. Leveraging these dynamical insights, we then propose \textit{Energy-Guarded Anderson Acceleration} (EG-AA), an algorithm that acts as an inertial governor to enforce thermodynamic consistency. Moreover, our convergence analysis, formulated via the \textbf{Acceleration Gain Factor}, proves that EG-AA improves upon gradient descent by maximizing the geometric contraction of the linear subspace projection while actively suppressing nonlinear approximation errors. Theoretical bounds confirm that EG-AA is no worse than standard AA, and numerical experiments demonstrate strictly improved convergence stability and rates in ill-conditioned convex composite problems compared to standard Anderson mixing.
\end{abstract}

\vspace{0.5cm}

\vspace{0.2cm}

\section{Introduction}
The optimization of objective functions is a fundamental challenge in computational mathematics. We consider the unconstrained minimization problem 
\begin{equation}
\min_{x \in \R^d} f(x).
\end{equation}
We define the fixed-point iteration based on Gradient Descent (GD) with step size 
$h$: 
\begin{equation}x_{k+1} =x_k - h \nabla f(x_k)=: g(x_k). \end{equation}
The residual is defined as: 
\begin{equation}
r(x) = g(x)-x=-h\nabla f(x). 
\end{equation}
While Anderson Acceleration (AA) is broadly applicable to fixed-point iterations, this work focuses specifically on its application to the gradient descent operator to leverage the potential structure of $f(x)$.

\subsection{Historical Context: From Gradient Descent to Acceleration}
The classic \textbf{Gradient Descent (GD)} method, $x^{k+1} = x^k - h\nabla f(x^k)$, dates back to Cauchy and is the simplest first-order oracle. Assuming that $\nabla f$ is $L$-Lipschitz continuous and that gradient descent uses a stepsize $h \le 1/L$ (or an appropriately decaying stepsize), the method attains the standard sublinear convergence rate for smooth convex objectives:
\[
f(x^k)-f(x^\ast)=O\!\left(\frac{1}{k}\right).
\]

In the 1980s, Nesterov \cite{Nesterov1983} introduced the \textbf{Accelerated Gradient (NAG)} method
which utilizes an inertial ``momentum'' term 
\[
    y^{k} = x^{k} + \beta_k (x^{k}-x^{k-1}),
\]
followed by a gradient step
\[
    x^{k+1} = y^{k} - h \nabla f(y^{k}),
\]
with the classical Nesterov choice 
\[
    \beta_k = \frac{t_{k-1}-1}{t_k}, 
    \qquad 
    t_{k+1} = \frac{1+\sqrt{1+4t_k^2}}{2}, 
    \quad t_0 = 1,
\]
which yields the accelerated convergence rate
\[
    f(x^{k}) - f(x^\ast) = O\!\left( \frac{1}{k^{2}} \right).
\]
This rate is proven to be optimal for the class of first-order smooth convex minimization algorithms.

Despite its optimality, NAG often exhibits significant transient oscillations (``ripples") in the function value and trajectory, particularly for ill-conditioned problems. These oscillations can be undesirable in applications requiring smooth trajectories or early stopping. Furthermore, the algebraic construction of Nesterov's momentum parameter $\beta_k$ was historically viewed as unintuitive.

\subsection{The Rise of ODE-Based Analysis and the Discretization Gap}
To demystify acceleration, researchers have turned to continuous-time limits. Su, Boyd, and Candès \cite{Su2016} famously showed that the continuous limit of NAG (as step size $h \to 0$) is a second-order Ordinary Differential Equation (ODE) with \textit{Asymptotic Vanishing Damping} (AVD). They Consider Nesterov's Accelerated Gradient method:
\begin{equation} \label{eq:nesterov_system}
    \begin{cases}
        x^{k+1} = y^k - h \nabla f(y^k), \\
        y^{k} = x^{k} + \frac{k-1}{k+2} (x^{k} - x^{k-1}).
    \end{cases}
\end{equation}
Eliminating the intermediate variable $y^k$, we obtain the equivalent single-variable update:
\begin{equation} \label{eq:nesterov_inertial}
    x^{k+1} - x^k = \frac{k-1}{k+2}(x^k - x^{k-1}) - h \nabla f(y^k).
\end{equation}
Let $t = k\sqrt{h}$ denote the continuous time variable, where $h$ is the step size and $k$ is the discrete iteration count. We assume there exists a smooth curve $x(t)$ such that as $h \to 0$ (and thus $t \to \infty$ effectively in the limit analysis), $x^k = x(k\sqrt{h}) = x(t)$. Applying the Taylor expansion yields:
\begin{align}
    x^{k+1} &= x((k+1)\sqrt{h}) = x(t + \sqrt{h}) = x(t) + \sqrt{h}\dot{x}(t) + \frac{h}{2}\ddot{x}(t) + o(h), \label{eq:taylor_next} \\
    x^{k-1} &= x((k-1)\sqrt{h}) = x(t - \sqrt{h}) = x(t) - \sqrt{h}\dot{x}(t) + \frac{h}{2}\ddot{x}(t) + o(h). \label{eq:taylor_prev}
\end{align}
Substituting these expansions into \eqref{eq:nesterov_inertial} and reorganizing terms, they derive the following limit equation (see Appendix A):
\begin{equation} \label{eq:intro_avd}
\ddot{x}(t) + \frac{3}{t}\dot{x}(t) + \nabla f(x(t)) = 0.
\end{equation}
This is the second-order Ordinary Differential Equation (ODE) corresponding to Nesterov's accelerated gradient method, often referred to as the Asymptotic Vanishing Damping (AVD) system. It captures the essence of acceleration: the friction term $\frac{3}{t}$ decays over time, allowing the system to maintain momentum initially and traverse flat regions of the landscape quickly. It has been proven that every trajectory generated by the AVD system satisfies the asymptotic convergence rate in function value:
\begin{equation}
    f(x(t)) - f^* = \mathcal{O}\left(\frac{1}{t^2}\right).
\end{equation}

However, recent advancements have identified a critical disconnect between continuous theory and discrete implementation. As highlighted by Xie et al. \cite{Xie2025}, the rapid convergence observed in continuous-time models (potentially faster than any inverse polynomial) often does not directly translate to discrete iterative methods. Naive explicit discretization of fast-converging ODEs can lead to divergence, particularly for problems with varying curvature. This necessitates a careful design of stable discretization schemes that can preserve the acceleration properties without succumbing to numerical instability.

\subsection{Hessian-Driven Damping (ISHD)}
To improve stability, recent work has focused on High-resolution ODEs and inertial systems with Hessian-driven damping (ISHD). Following the framework of Attouch et al.~\cite{Attouch2020} and its generalization by Xie et al.~\cite{Xie2025}, the ISHD system is given by:
\begin{equation} \label{eq:intro_din}
\ddot{x}(t) + \frac{\alpha}{t}\dot{x}(t) + \beta(t) \nabla^2 f(x(t))\dot{x}(t) + \gamma(t)\nabla f(x(t)) = 0.
\end{equation}
The term $\beta(t) \nabla^2 f(x(t))\dot{x}(t)$ represents geometric damping derived from the Hessian. Crucially, this term allows for the control of oscillations without affecting the asymptotic convergence rate. Since $\nabla^2 f(x)\dot{x} = \frac{d}{dt} \nabla f(x)$, this damping can be implemented in discrete algorithms using first-order differences $(\nabla f(x^k) - \nabla f(x^{k-1}))$, avoiding expensive Hessian computations. Xie et al. recently demonstrated that with appropriate stability conditions on $\beta(t)$ and $\gamma(t)$, explicit discretization schemes (such as their EIGAC algorithm) can indeed preserve the $\mathcal{O}(1/k^2)$ convergence rate while ensuring robustness \cite{Xie2025}.

\subsection{From ISHD to Anderson Acceleration}
The ISHD framework provides a unified view, but selecting the optimal coefficients $\alpha, \beta(t), \gamma(t)$ for a given problem remains a non-trivial challenge. The ``Learning to Optimize" (L2O) paradigm \cite{Xie2025} addresses this by training these coefficients over a distribution of functions to minimize a stopping time metric.

AA can be viewed as the \textit{online, adaptive} counterpart to this learning process. Instead of learning coefficients offline from a training set, AA adapts its mixing coefficients at every iteration based on the local residual history. While classical Nesterov acceleration is well-understood via asymptotic vanishing damping, the dynamics of Anderson Acceleration (AA) remain less transparent.

\textbf{The central questions of this paper are:} What is the precise continuous-time physical model that describes Anderson Acceleration within the \textbf{High-Resolution ODE} framework? How does the energy dissipation of this system govern its convergence, and can we leverage these physical insights to construct an acceleration method that is robustly superior to standard AA?

To answer these questions, we proceed in four rigorous stages:
\begin{enumerate}
    \item \textbf{Discrete Algebraic Transformation:} We systematically rewrite the AA update as an adaptive multi-step momentum method. This fundamental step converts the opaque residual minimization coefficients ($\theta$) into interpretable inertial momentum coefficients ($\gamma$), establishing the structural link to classical acceleration.
    
    \item \textbf{Continuous Limit Analysis:} We apply Taylor expansions under the inertial scaling regime to derive the specific ODEs governing the system. We prove that AA converges to a second-order ODE with \textbf{Variable Effective Mass}, where the mass is dynamically determined by the second moment of the mixing coefficients.
    
    \item \textbf{Energy Dissipation and Stability Analysis:} We derive the Lyapunov energy decay rate for the AA($m$) ODE limit. This analysis reveals a profound physical insight invisible in standard fixed-mass momentum methods: instability in AA arises when the algorithm increases effective mass too rapidly, effectively injecting energy into the system and violating passivity constraints.
    
    \item \textbf{A Novel Energy-Guarded Algorithm:} Finally, leveraging these dynamical insights, we propose \textit{Energy-Guarded Anderson Acceleration} (EG-AA). This algorithm acts as a physical governor, explicitly controlling the growth of effective mass to ensure continuous energy dissipation. We demonstrate that EG-AA achieves strictly improved convergence stability and rates in ill-conditioned convex composite problems compared to standard Anderson mixing.
\end{enumerate}

\section{Preliminaries}

\subsection{Anderson Acceleration}
Anderson acceleration (AA), originally introduced by Anderson in 1965 to accelerate the convergence of nonlinear fixed-point iterations \cite{anderson1965iterative, anderson2019comments}, has become a standard tool in scientific computing.
It is widely utilized in computational chemistry as Pulay mixing and in electronic structure calculations as Anderson mixing.
In contrast to the standard Picard iteration, which relies solely on the most recent iterate, Anderson acceleration of depth $m$, denoted as $AA(m)$, constructs the next update using a linear combination of the $m$ previous iterates.
This combination is determined by minimizing the linearized residual of the fixed-point map in a least-squares sense.
Theoretical analysis has established that AA is essentially equivalent to the nonlinear GMRES method \cite{carlson1998design, oosterlee2000krylov, walker2011anderson} and is closely related to direct inversion in the iterative subspace (DIIS) methods \cite{pulay1980convergence}.
Furthermore, it shares structural similarities with quasi-Newton update schemes \cite{fang2009two, haelterman2010similarities}.
Recent research has focused on extending the classical algorithm to non-stationary variants.
Evans et al. \cite{evans2020proof} introduced heuristics for dynamically selecting damping factors based on iteration gain, while Pollock and Rebholz \cite{pollock2021anderson} proposed strategies for adaptively varying the window size.
Inspired by nested Krylov solvers such as GMRESR \cite{vuik1993solution, van1994gmresr}, Chen and Vuik \cite{chen2022composite} developed a composite Anderson acceleration method that integrates two distinct window sizes to enhance robustness.

While the empirical effectiveness of AA has been recognized for decades, rigorous convergence theory has only been developed relatively recently.
Toth and Kelley \cite{toth2015convergence} provided the first proof of local $r$-linear convergence for stationary, undamped AA applied to contractive maps.
This result was subsequently extended by Evans et al. \cite{evans2020proof} to include damping factors.
In the context of partial differential equations, Pollock et al. \cite{pollock2019anderson} demonstrated that AA improves the convergence rate of Picard iterations for the steady incompressible Navier-Stokes equations.
De Sterck and He \cite{sterck2021asymptotic} generalized these findings to broader fixed-point iterations $\mathbf{x}=\mathbf{g}(\mathbf{x})$, linking convergence behavior to the spectrum of the Jacobian $\mathbf{g}'(\mathbf{x})$ at the fixed point.
Additionally, Wang et al. \cite{wang2021asymptotic} analyzed the asymptotic linear convergence of AA when applied to the Alternating Direction Method of Multipliers (ADMM).
Despite these advancements, establishing global convergence guarantees and sharpening local convergence rates remain active areas of research.
For further details on AA and its diverse applications, we refer readers to \cite{brune2015composing, toth2017local, peng2018anderson, zhang2019accelerating, zhang2020globally, bian2021anderson, chen2022non,khatiwala2023fast} and the references therein.

\subsection{Type-II Anderson Acceleration}

In this work we focus on the Type-II variant of Anderson Acceleration (AA). 
Let $g$ denote a fixed-point map of Gradient Descent iteration as in previous section and write the residual at iterate $x_k$
\[
r_k = g(x_k)-x_k=-h\nabla f(x_k).
\]

\paragraph{Difference matrices.}
At iteration $k$ with memory depth $m_k=\min(m,k)$, Type-II AA forms the difference vectors
\[
\Delta x_j := x_{j+1}-x_j,\qquad \Delta r_j := r_{j+1}-r_j,
\]
and collects them in the matrices
\[
X_k := [\,\Delta x_{k-m_k},\dots,\Delta x_{k-1}\,]\in\R^{d\times m_k},\qquad
R_k := [\,\Delta r_{k-m_k},\dots,\Delta r_{k-1}\,]\in\R^{d\times m_k}.
\]

\paragraph{Type-II least-squares characterization.}
The coefficients $\theta^{(k)}\in\R^{m_k}$ are obtained by solving the (typically overdetermined) least-squares problem
\begin{equation}\label{eq:ls-problem-typedetail}
    \theta^{(k)}_{\mathrm{II}}
    = \arg\min_{\theta\in\R^{m_k}}
        \big\| r_k - R_k \theta \big\|_2^2,
\end{equation}
which minimizes the linearized residual prediction error.  
The associated normal equations are
\[
R_k^\top R_k\,\theta^{(k)}_{\mathrm{II}} = R_k^\top r_k .
\]
When $R_k$ is rank-deficient or ill-conditioned, one commonly replaces \eqref{eq:ls-problem-typedetail} by the Tikhonov-regularized problem
\[
\theta^{(k)}_{\lambda}
    = \arg\min_{\theta\in\R^{m_k}}
      \big\| r_k - R_k \theta\big\|_2^2 + \lambda\|\theta\|_2^2,
\]
or solves \eqref{eq:ls-problem-typedetail} using a QR/SVD-based minimum-norm solver;  
the parameter $\lambda\ge0$ provides additional damping and improves numerical stability.

Given $\theta^{(k)}_{\mathrm{II}}$, the Type-II AA update without damping is
\begin{equation}\label{eq:aa-update-typedetail}
    x_{k+1}^{\mathrm{II}}
    = x_k + r_k - (X_k + R_k)\,\theta^{(k)}_{\mathrm{II}} .
\end{equation}
Equivalently, the update may be expressed as an affine combination of past iterates; see  
\cite{walker2011anderson, fang2009two} for equivalent algebraic formulations.

\paragraph{Comparison with Type-I AA.}
For completeness, we recall that the Type-I variant instead solves the least-squares system
\begin{equation}\label{eq:typeI-def}
    \theta^{(k)}_{\mathrm{I}}
    = \arg\min_{\theta\in\R^{m_k}}
        \big\| \Delta x_k - X_k \theta \big\|_2^2,
\end{equation}
and performs the update
\begin{equation}\label{eq:typeI-update}
    x_{k+1}^{\mathrm{I}}
    = x_k + r_k - X_k\,\theta^{(k)}_{\mathrm{I}} .
\end{equation}
Thus, Type-II minimizes the discrepancy in the residual space, whereas Type-I minimizes the predicted
iterate increment.  
When the Jacobian $g'(x)$ is constant, the identity
\[
R_k = g'(x^\ast)\, X_k
\]
holds, implying that Type-I and Type-II are algebraically equivalent up to a fixed left multiplication.  
In general nonlinear settings, however, the two variants may differ substantially in conditioning and 
numerical stability.

\paragraph{Remarks on implementation and theory.}
(i) The least-squares system can be ill-conditioned when the columns of $R_k$ are nearly linearly dependent; thus regularization, damping, or truncating the window size $m_k$ are standard remedies.  
(ii) The computational cost is dominated by forming $R_k^\top R_k$ (or its QR factorization) and solving a small $m_k\times m_k$ system; in practice one typically chooses $m_k\le 10$ to balance overhead and acceleration.  
(iii) Many local convergence results for AA rely on linearization of $g$ about the fixed point or contractivity assumptions; such assumptions will be explicitly stated when invoked in the subsequent analysis.

\section{Momentum Transformation: From Discrete Anderson Mixing to Adaptive Momentum}

In this section, we rigorously transform the Anderson Acceleration (AA) update rule into a multi-step momentum form. This algebraic reformulation is not merely a change of variables; it is essential for revealing the underlying inertial structure required for the subsequent continuous-time ODE analysis.

\begin{figure}[t]
    \centering
    \begin{tikzpicture}[scale=1.1, >=Stealth]
        \node (yk) at (0,0) [circle, fill=blue!15, draw, inner sep=2pt] {\footnotesize $y_k$};
        \node (yk1) at (3,1) [circle, fill=blue!15, draw, inner sep=2pt] {\footnotesize $y_{k+1}$};
        \node (yk2) at (5,2.5) [circle, fill=blue!15, draw, inner sep=2pt] {\footnotesize $y_{k+2}$};
        
        \draw[->, thick, red!80!black] (yk) -- (yk1) node[midway, below right] {$s_1$};
        \draw[->, thick, red!80!black] (yk1) -- (yk2) node[midway, above left] {$s_0$};
        
        \draw[->, dashed, blue!80!black] (yk) -- (yk2) node[midway, above, sloped] {\footnotesize $y_{k+2} - y_k$};
        
        \node[text width=4cm, align=center, fill=white, inner sep=2pt, draw=gray!20] at (1.5, 2.5) {\small \textbf{Basis Transformation}\\$\{s_j\} \leftrightarrow \{y_{t} - y_{t-j}\}$};
    \end{tikzpicture}
    \caption{Visualization of the basis transformation. AA is naturally defined using sequential difference steps $s_j$ (solid red). The momentum form requires rewriting these in terms of lags from the current position (dashed blue), which correspond to velocity and acceleration terms in the continuous limit.}
    \label{fig:basis_transform}
\end{figure}
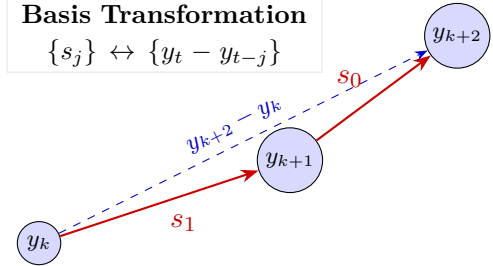

A crucial component of our transformation is the auxiliary variable $y_{k+1}$, defined as the standard gradient descent step from $x_k$:
\begin{equation} \label{eq:y_definition}
    y_{k+1} := x_k + r(x_k) = x_k - h \nabla f(x_k).
\end{equation}
Note that in our analysis, the step size $h > 0$ is absorbed into the definition of the residual $r(x) = -h\nabla f(x)$.

Using this definition, we can derive a fundamental identity relating the columns of the matrix $X_k + R_k$ in the AA update to the sequence $\{y_k\}$. For the $j$-th column corresponding to iteration $k-j$ (with $j = 1, \dots, m_k$ where $m_k = \min(m, k)$), we have:
\begin{align}
    (X_k + R_k)_{:,j} &= \Delta x_{k-m_k+j-1} + \Delta r_{k-m_k+j-1} \nonumber \\
    &= (x_{k-m_k+j} - x_{k-m_k+j-1}) + (r(x_{k-m_k+j}) - r(x_{k-m_k+j-1})) \nonumber \\
    &= (x_{k-m_k+j} + r(x_{k-m_k+j})) - (x_{k-m_k+j-1} + r(x_{k-m_k+j-1})) \nonumber \\
    &= y_{k-m_k+j+1} - y_{k-m_k+j}.
    \label{eq:column_identity}
\end{align}

For notational convenience, we define the sequential gradient steps as:
\begin{equation}
    s_j := y_{k-j+2} - y_{k-j+1} = \Delta y_{k-j+1}, \quad j = 1, 2, \dots, m.
    \label{eq:sj_definition}
\end{equation}
Note that $s_1 = y_{k+1} - y_k$, $s_2 = y_k - y_{k-1}$, and in general $s_j = y_{k-j+2} - y_{k-j+1}$. 
When considering a memory depth $m$, the AA update involves the terms $s_{m}, s_{m-1}, \dots, s_1$. Consequently, the linear combination term in the AA update can be expressed compactly as:
\begin{equation}
    \sum_{j=1}^m s_{m-j+1} \theta_j^{(k)} = \sum_{j=1}^m s_j \theta_{m-j+1}^{(k)}.
    \label{eq:linear_combination}
\end{equation}

With this foundation, we can now derive the momentum form for varying depths $m$. Consider first the case of depth $m=1$. The update involves only the most recent difference $s_1 = y_{k+1} - y_k$, leading to:
\begin{align}
    x_{k+1} &= (x_k + r(x_k)) - (\Delta x_{k-1} + \Delta r_{k-1}) \theta^{(k)} \nonumber \\
    &= y_{k+1} - s_1 \theta^{(k)}.
    \label{eq:m1_update}
\end{align}
By defining the momentum coefficient $\gamma_k := -\theta^{(k)}$, where $\theta^{(k)}$ is a scalar for $m=1$, the update transforms into the recognizable Adaptive Heavy-Ball form:
\begin{equation}
     x_{k+1} = y_{k+1} + \gamma_k (y_{k+1} - y_k). 
     \label{eq:m1_momentum}
\end{equation}
Here, the coefficient $\gamma_k$ is adaptive, calculated as $\gamma_k = -\frac{\langle \Delta r_{k-1}, r_k \rangle}{\|\Delta r_{k-1}\|^2}$ to minimize the residual.

Next, we extend this logic to depth $m=2$. The AA update linear combination becomes:
\[
\text{Correction} = - \theta_1^{(k)} s_2 - \theta_2^{(k)} s_1.
\]
We seek to rewrite this in a two-step momentum basis:
\begin{equation} \label{eq:momentum_basis_2}
    \text{Momentum form: } \gamma^{(1)}(y_{k+1} - y_k) + \gamma^{(2)}(y_{k+1} - y_{k-1}).
\end{equation}
Expressing the basis terms via $s_j$, we have $y_{k+1} - y_k = s_1$ and $y_{k+1} - y_{k-1} = s_1 + s_2$. 
Equating the coefficients of $s_1$ and $s_2$ in the AA update and the momentum form yields the linear system:
\begin{align*}
    s_1: & \quad -\theta_2^{(k)} = \gamma^{(1)} + \gamma^{(2)}, \\
    s_2: & \quad -\theta_1^{(k)} = \gamma^{(2)}.
\end{align*}
Solving this system provides the mapping: $\gamma^{(2)} = -\theta_1^{(k)}$ and $\gamma^{(1)} = \theta_1^{(k)} - \theta_2^{(k)}$.

Proceeding similarly for $m=3$, the update term $- (\theta_1^{(k)} s_3 + \theta_2^{(k)} s_2 + \theta_3^{(k)} s_1)$ is matched against the three-step momentum basis $\sum_{j=1}^3 \gamma^{(j)} (y_{k+1} - y_{k-j+1})$. 
By expanding the lag terms, we obtain the system of equations:
\begin{align*}
    s_3: & \quad \gamma^{(3)} = -\theta_1^{(k)}, \\
    s_2: & \quad \gamma^{(2)} + \gamma^{(3)} = -\theta_2^{(k)}, \\
    s_1: & \quad \gamma^{(1)} + \gamma^{(2)} + \gamma^{(3)} = -\theta_3^{(k)}.
\end{align*}
The solution reveals a consistent recursive pattern: 
\begin{align*}
    \gamma^{(3)} &= -\theta_1^{(k)}, \\
    \gamma^{(2)} &= \theta_1^{(k)} - \theta_2^{(k)}, \\
    \gamma^{(1)} &= \theta_2^{(k)} - \theta_3^{(k)}.
\end{align*}

Generalizing this pattern by induction leads to our main transformation theorem.

\begin{theorem}[AA-Momentum Equivalence] \label{thm:aa_momentum}
The Type-II AA($m$) update is algebraically equivalent to the adaptive multistep momentum method:
\begin{equation}
    x_{k+1} = y_{k+1} + \sum_{j=1}^m \gamma_k^{(j)} (y_{k+1} - y_{k-j+1}),
    \label{eq:momentum_form}
\end{equation}
where $\theta^{(k)} = (\theta_1^{(k)}, \theta_2^{(k)}, \dots, \theta_m^{(k)})^\top$ are the AA mixing coefficients obtained from solving \eqref{eq:ls-problem-typedetail}, and the momentum coefficients are derived via the difference relation:
\begin{equation}
    \gamma_k^{(j)} = \theta_{m-j}^{(k)} - \theta_{m-j+1}^{(k)}, \quad \text{for } j=1,\dots,m,
    \label{eq:gamma_general}
\end{equation}
with the convention that $\theta_0^{(k)} = 0$ and $\theta_{m+1}^{(k)} = 0$ as boundary conditions for the recurrence.
\end{theorem}

\begin{proof}
The proof proceeds by induction on $m$. The base cases $m=1,2,3$ have been verified above. 
Assume the formula holds for memory depth $m-1$. For depth $m$, write the AA correction term as:
\[
-\sum_{i=1}^{m} \theta_i^{(k)} s_{m-i+1}.
\]
The momentum form seeks coefficients $\gamma^{(1)}, \dots, \gamma^{(m)}$ such that:
\[
\sum_{j=1}^{m} \gamma^{(j)} (y_{k+1} - y_{k-j+1}) = \sum_{j=1}^{m} \gamma^{(j)} \sum_{i=1}^{j} s_i.
\]
Equating coefficients for $s_1, \dots, s_m$ yields the triangular system:
\[
\sum_{j=i}^{m} \gamma^{(j)} = -\theta_{m-i+1}^{(k)}, \quad i=1,\dots,m.
\]
Solving this system recursively from $i=m$ down to $i=1$ gives:
\[
\gamma^{(m)} = -\theta_1^{(k)}, \quad \gamma^{(j)} = \theta_{m-j}^{(k)} - \theta_{m-j+1}^{(k)} \text{ for } j=1,\dots,m-1,
\]
which is exactly \eqref{eq:gamma_general} with the boundary conditions.
\end{proof}

This theorem is pivotal for our analysis. It demonstrates that Anderson Acceleration effectively constructs a velocity vector using a weighted history of displacements $y_{k+1} - y_{k-j+1}$, which directly maps to velocity and acceleration terms in the continuous-time limit. 
Moreover, the transformation reveals that AA's mixing coefficients $\theta_i^{(k)}$ encode not just momentum but also higher-order inertial information through their successive differences.

\section{Continuous-Time Analysis: The Limit ODE}

\label{sec:continuous_derivation}

We now derive the corresponding continuous-time limit of Anderson Acceleration using the standard Nesterov scaling as in \cite{Shi2018,Su2016} . Let $h > 0$ be the discretization step size, and define the scaled time as $t_k = k\sqrt{h}$. This scaling is essential for capturing acceleration phenomena in the continuous limit.

For the gradient descent step, we adopt the scaling $\beta = h$, consistent with the Euler discretization of gradient flow. Thus:
\begin{equation}
    y_{k+1} = x_k + r(x_k) = x_k - h\nabla f(x_k).
    \label{eq:y_definition_scaled}
\end{equation}

Recall from Theorem \ref{thm:aa_momentum} that AA($m$) is algebraically equivalent to:
\begin{equation}
    x_{k+1} = y_{k+1} + \sum_{j=1}^m \gamma_k^{(j)} (y_{k+1} - y_{k-j+1}),
    \label{eq:momentum_form_scaled}
\end{equation}
where $\gamma_k^{(j)}$ are the adaptive momentum coefficients.

\subsection{Asymptotic Expansion under Nesterov Scaling}

We assume that as the step size $h \to 0$, the discrete iterates $x_k$ approximate a smooth trajectory $x(t)$ with $x_k \approx x(t_k)$. Under this high-resolution scaling, the iteration count $k$ relates to the continuous time $t$ via the square root of the step size:
\begin{itemize}
    \item \textbf{Time increment:} The discrete step corresponds to a time interval $\Delta t = t_{k+1} - t_k = \sqrt{h}$.
    \item \textbf{Lagged steps:} A delay of $j$ iterations, denoted by $y_{k-j+1}$, corresponds to a time lag of $j\sqrt{h}$ relative to $t_{k+1}$.
\end{itemize}

We now perform Taylor expansions around the current time $t = t_k$. We first analyze the gradient step $y_{k+1}$. Substituting the continuous ansatz $x_k = x(t)$ into the discrete update rule, we obtain:
\begin{align}
    y_{k+1} &= x_k - h\nabla f(x_k) \nonumber \\
            &= x(t) - h\nabla f(x(t)) + \mathcal{O}(h^{3/2}).
    \label{eq:y_expansion}
\end{align}
The inclusion of the error term $\mathcal{O}(h^{3/2})$ is necessary to ensure consistency with the expansion of lagged terms in the subsequent analysis. Specifically, when handling momentum or lagged states (e.g., $y_{k-j+1}$), we perform Taylor expansions on terms such as $x(t - \Delta t)$:
\begin{equation}
    x(t-\Delta t) = x(t) - \Delta t \dot{x}(t) + \frac{1}{2}\Delta t^2 \ddot{x}(t) - \mathcal{O}(\Delta t^3).
\end{equation}
Noting that under our scaling $\Delta t^3 = (\sqrt{h})^3 = h^{3/2}$, it becomes evident that capturing the second-order derivative $\ddot{x}(t)$ requires retaining terms up to $\mathcal{O}(\Delta t^2) = \mathcal{O}(h)$. Consequently, terms of order $\mathcal{O}(\Delta t^3)$ correspond exactly to $\mathcal{O}(h^{3/2})$ and represent the truncation error for the limiting second-order ODE.

For the momentum lag terms $y_{k+1} - y_{k-j+1}$, the time difference is $j\sqrt{h}$. We expand:
\begin{align}
    y_{k+1} - y_{k-j+1} &= [x(t) - h\nabla f(x(t))] - [x(t - j\sqrt{h}) - h\nabla f(x(t - j\sqrt{h}))] \nonumber \\
    &= [x(t) - x(t - j\sqrt{h})] - h[\nabla f(x(t)) - \nabla f(t - j\sqrt{h})] + \mathcal{O}(h^{3/2}).
\end{align}

Using the Taylor expansion $x(t - j\sqrt{h}) = x(t) - j\sqrt{h}\dot{x}(t) + \frac{j^2 h}{2}\ddot{x}(t) + \mathcal{O}(h^{3/2})$, and noting that the gradient difference term is $\mathcal{O}(h \cdot \sqrt{h}) = \mathcal{O}(h^{3/2})$, we obtain:
\begin{align}
    y_{k+1} - y_{k-j+1} &= j\sqrt{h}\dot{x}(t) - \frac{j^2 h}{2}\ddot{x}(t) + \mathcal{O}(h^{3/2}).
    \label{eq:lag_expansion_final}
\end{align}

First, we matching the momentum update. Substitute \eqref{eq:y_expansion} and \eqref{eq:lag_expansion_final} into the momentum form \eqref{eq:momentum_form_scaled}:
\begin{align}
    x_{k+1} &= [x(t) - h\nabla f(x(t))] \nonumber \\
    &\quad + \sum_{j=1}^m \gamma_k^{(j)} \left[ j\sqrt{h}\dot{x}(t) - \frac{j^2 h}{2}\ddot{x}(t) \right] + \mathcal{O}(h^{3/2}).
    \label{eq:full_expansion}
\end{align}

We also expand the LHS $x_{k+1} = x(t+\sqrt{h})$:
\begin{align}
    x(t+\sqrt{h}) &= x(t) + \sqrt{h}\dot{x}(t) + \frac{h}{2}\ddot{x}(t) + \mathcal{O}(h^{3/2}).
    \label{eq:x_forward_expansion}
\end{align}

Equating \eqref{eq:full_expansion} and \eqref{eq:x_forward_expansion}:
\begin{align}
    x(t) + \sqrt{h}\dot{x} + \frac{h}{2}\ddot{x} &= x(t) - h\nabla f + \sum_{j=1}^m \gamma_k^{(j)} \left[ j\sqrt{h}\dot{x} - \frac{j^2 h}{2}\ddot{x} \right] + \mathcal{O}(h^{3/2}).
    \label{eq:equality}
\end{align}

Next, we derive the consistency condition from $\mathcal{O}(\sqrt{h})$ terms. Matching terms of order $\sqrt{h}$ (velocity matching):
\begin{align}
    \sqrt{h}\dot{x}(t) &= \sum_{j=1}^m \gamma_k^{(j)} j\sqrt{h}\dot{x}(t).
\end{align}

Dividing by $\sqrt{h}\dot{x}(t)$, we recover the \textbf{Kinematic Consistency Condition},
\begin{equation}
    \sum_{j=1}^m j\gamma_k^{(j)} = 1 + \mathcal{O}(\sqrt{h}),
    \label{eq:consistency_scaled}
\end{equation}
which is precisely connection to Anderson Acceleration consistency condition in \cite{walker2011anderson} that the AA coefficients must satisfy a
first-order Taylor compatibility condition so that the least-squares correction does not distort the local linearization of the fixed-point map, see Remark \ref{rem:WN_connection}.

We define the damping coefficient $c(t)$ via the deviation from this condition:
\begin{equation}
    \sum_{j=1}^m j\gamma_k^{(j)} = 1 - \sqrt{h} c(t).
    \label{eq:consistency_with_damping}
\end{equation}

Lastly, we match the terms of order $h$ (acceleration matching) in \eqref{eq:equality}:
\begin{align}
    \frac{h}{2}\ddot{x} &= -h\nabla f - \sum_{j=1}^m \gamma_k^{(j)} \frac{j^2 h}{2}\ddot{x} + \text{(terms from consistency deviation)}.
\end{align}
Substituting the damping relation $\sum j\gamma \sqrt{h}\dot{x} = (1 - \sqrt{h}c)\sqrt{h}\dot{x}$ into the $\mathcal{O}(\sqrt{h})$ match leaves a residual of $-h c(t)\dot{x}$ on the RHS. Moving this to the LHS and rearranging:

\begin{align}
    \frac{h}{2}\left[ 1 + \sum_{j=1}^m j^2 \gamma_k^{(j)} \right] \ddot{x}(t) + h c(t)\dot{x}(t) + h \nabla f(x(t)) = 0.
\end{align}

Dividing by $h$, we obtain the final theorem.

\begin{theorem}[Variable Mass ODE for Anderson Acceleration]\label{thm:variable_mass}
Under the Nesterov scaling $t_k = k\sqrt{h}$, the continuous-time limit of AA($m$) as $h \to 0$ is the following second-order ODE with \textbf{variable effective mass}:
\begin{equation}
    M_{\text{eff}}(t) \ddot{x}(t) + c(t) \dot{x}(t) + \nabla f(x(t)) = 0,
    \label{eq:variable_mass_ode}
\end{equation}
where the \textbf{effective mass} is determined by the second moment of the momentum coefficients:
    \begin{equation}
        M_{\text{eff}}(t) = \lim_{h\to 0}\frac{1}{2} \left( 1 + \sum_{j=1}^m j^2 \gamma^{(j)}(t) \right).
        \label{eq:effective_mass}
    \end{equation}
    
The \textbf{damping coefficient} $c(t)$ captures the rate of consistency violation:
    \begin{equation}
        c(t) = \lim_{h\to 0} \frac{1}{\sqrt{h}} \left( 1 - \sum_{j=1}^m j\gamma_k^{(j)} \right).
        \label{eq:damping_coefficient}
    \end{equation}
\end{theorem}

\begin{remark}[Connection to Walker-Ni \cite{walker2011anderson}, Remark~2.7]
\label{rem:WN_connection}
The scaled consistency condition
\begin{equation}
    \sum_{j=1}^m j\gamma_k^{(j)}
    = 1 + \mathcal{O}(\sqrt{h})
    \tag{\ref{eq:consistency_scaled}}
\end{equation}
is the continuous-time analogue of the discrete first-order
residual consistency described in Walker-Ni \cite{walker2011anderson}.  In their setting, the Anderson coefficients must satisfy a
first-order Taylor compatibility condition so that the least-squares correction does not distort the local linearization
of the fixed-point map.  In practice, the least-squares solution for $\gamma_k^{(j)}$ in Anderson Acceleration does not enforce this equality exactly. Moreover, the linear equivalence of AA to GMRES (Walker–Ni \cite{walker2011anderson}, Theorem~2.2) implies that the condition holds only approximately, and may fail during stagnation (Remark~2.7). Hence, the condition serves as a continuous-time guide rather than a strict discrete constraint.

Our condition enforces the same principle on the time scale $h = \Delta t$:  the weighted memory coefficients $\gamma_k^{(j)}$ reproduce the
first-order kinematic relation
\[
x(t+h) - x(t) = h\,\dot{x}(t) + \mathcal{O}(h^{3/2}),
\]
and therefore the linear combination
$\sum_{j=1}^m j\gamma_k^{(j)}$ must converge to~1 as $h\to0$. The $\mathcal{O}(\sqrt{h})$ deviation reflects the stochastic time-scale of the implicit EM discretization and the fact that the least-squares system is perturbed by residual noise. Thus, \eqref{eq:consistency_scaled} is precisely the Walker-Ni first-order compatibility condition, adapted to the continuous-time limits of Anderson acceleration.
\end{remark}

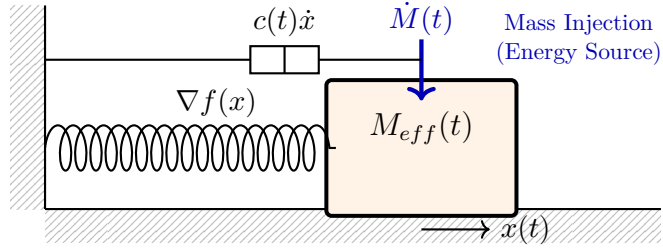
\begin{figure}[t]
    \centering
    \begin{tikzpicture}[scale=0.9]
        \tikzstyle{spring}=[decorate,decoration={coil,aspect=0.4,segment length=2mm,amplitude=3mm}, thick]
        
        \fill[pattern=north east lines, pattern color=gray!50] (-1,-0.5) rectangle (8,0);
        \draw[thick] (-1,0) -- (8,0); 
        \fill[pattern=north east lines, pattern color=gray!50] (-1.5,0) rectangle (-1,3);
        \draw[thick] (-1,0) -- (-1,3); 

        \node[draw, line width=1.5pt, minimum width=2.5cm, minimum height=1.8cm, fill=orange!10, rounded corners=2pt] (mass) at (4.5, 0.9) {};
        
        \node at (4.5, 1.2) {\large $M_{eff}(t)$};
        \draw[->, ultra thick, blue!70!black] (4.5, 2.5) -- (4.5, 1.6);
        \node[blue!70!black] at (4.5, 2.8) {$\dot{M}(t)$};
        \node[text width=3cm, align=center, scale=0.8, blue!70!black] at (6.8, 2.5) {Mass Injection\\(Energy Source)};

        \draw[spring] (-1, 0.9) -- (3.25, 0.9); 
        \node[above] at (1.5, 1.2) {$\nabla f(x)$};

        \draw[thick] (-1, 2.2) -- (2.0, 2.2); 
        \draw[thick] (2.0, 2.0) rectangle (3.0, 2.4); 
        \draw[thick] (3.0, 2.2) -- (4.5, 2.2); 
        \draw[thick] (2.5, 2.0) -- (2.5, 2.4); 
        \node[above] at (2.5, 2.4) {$c(t) \dot{x}$};
        
        \draw[->, thick] (4.5, -0.3) -- (5.5, -0.3) node[right] {$x(t)$};
    \end{tikzpicture}
    \caption{Physical Model of Anderson Acceleration. The system behaves as a damped oscillator with \textbf{variable mass}. The rate of change of mass $\dot{M}$ acts as an energy source or sink, explaining both the acceleration capability and potential instability of the algorithm.}
    \label{fig:variable_mass_model}
\end{figure}

\subsection{Impact of Memory Depth on Effective Mass}

The derived formula \eqref{eq:effective_mass} reveals how the memory depth $m$ explicitly alters the system's inertia. We analyze specific cases by imposing the kinematic consistency condition $\sum_{j=1}^m j \gamma^{(j)} = 1$:

\begin{itemize}
    \item \textbf{AA(1) (Standard Nesterov Mass):} 
    For depth $m=1$, the consistency condition forces $1 \cdot \gamma^{(1)} = 1 \implies \gamma^{(1)} = 1$. Substituting this into the mass formula:
    \begin{equation}
        M_{\text{eff}} = \frac{1}{2}\left(1 + 1^2 \cdot 1\right) = 1.
    \end{equation}
    This confirms that AA(1) recovers the constant unit mass dynamics, equivalent to the standard Nesterov acceleration limit.

    \item \textbf{AA(2) (Augmented Inertia):} 
    For depth $m=2$, consistency requires $\gamma^{(1)} + 2\gamma^{(2)} = 1$, which implies $\gamma^{(1)} = 1 - 2\gamma^{(2)}$. Substituting this into the second moment sum:
    \begin{align*}
        \sum_{j=1}^2 j^2 \gamma^{(j)} &= 1^2(1 - 2\gamma^{(2)}) + 2^2 \gamma^{(2)} \\
        &= 1 - 2\gamma^{(2)} + 4\gamma^{(2)} = 1 + 2\gamma^{(2)}.
    \end{align*}
    The effective mass becomes:
    \begin{equation}
        M_{\text{eff}} = \frac{1}{2}(1 + [1 + 2\gamma^{(2)}]) = 1 + \gamma^{(2)}.
    \end{equation}
    This shows that the second momentum coefficient $\gamma^{(2)}$ linearly augments the mass. If $\gamma^{(2)} > 0$, the system becomes ``heavier'' than Nesterov.

    \item \textbf{AA(3) (Deep Memory Effect):} 
    For depth $m=3$, using $\gamma^{(1)} = 1 - 2\gamma^{(2)} - 3\gamma^{(3)}$, the second moment sum expands to:
    \begin{align*}
        \sum_{j=1}^3 j^2 \gamma^{(j)} &= (1 - 2\gamma^{(2)} - 3\gamma^{(3)}) + 4\gamma^{(2)} + 9\gamma^{(3)} \\
        &= 1 + 2\gamma^{(2)} + 6\gamma^{(3)}.
    \end{align*}
    Substituting back into \eqref{eq:effective_mass}:
    \begin{equation}
        M_{\text{eff}} = \frac{1}{2}(2 + 2\gamma^{(2)} + 6\gamma^{(3)}) = 1 + \gamma^{(2)} + 3\gamma^{(3)}.
    \end{equation}
    Notably, the coefficient of the deepest term $\gamma^{(3)}$ is scaled by a factor of 3. This indicates that older historical information (larger lag $j$) has a disproportionately large, super-linear impact on the system's effective inertia.
\end{itemize}

\section{Continuous-Time Analysis: The High-Resolution Limit ODE}
\label{sec:continuous_derivation_hr}

To fully understand the dynamical behavior of Anderson Acceleration, particularly its stability properties in stiff regimes, the limiting analysis $h \to 0$ is insufficient. We must perform a \textbf{High-Resolution Analysis} by retaining terms of higher order in $\sqrt{h}$. This approach, pioneered for Nesterov acceleration by Shi et al. \cite{Shi2018}, reveals the implicit regularization mechanisms that vanish in the continuous limit but govern the discrete algorithm's behavior.

\subsection{Detailed Asymptotic Expansion}

We proceed with the Nesterov scaling $t_k = k\sqrt{h}$. The discrete update is given by:
\begin{equation}
    x_{k+1} = y_{k+1} + \sum_{j=1}^m \gamma_k^{(j)} (y_{k+1} - y_{k-j+1}),
    \label{eq:hr_master_update}
\end{equation}
with $y_{k+1} = x_k - h\nabla f(x_k)$.

\subsubsection{Expansion of the Lag Terms}
The critical difference between standard momentum and Anderson acceleration lies in the structure of the lag terms $y_{k+1} - y_{k-j+1}$. Unlike simple position differences, these terms contain gradient information.
Recall the definition:
\begin{align}
    y_{k+1} - y_{k-j+1} &= [x_k - h\nabla f(x_k)] - [x_{k-j} - h\nabla f(x_{k-j})] \nonumber \\
    &= \underbrace{(x_k - x_{k-j})}_{\text{Displacement}} - h \underbrace{[\nabla f(x_k) - \nabla f(x_{k-j})]}_{\text{Gradient Difference}}.
    \label{eq:lag_structure}
\end{align}

We expand both components around time $t$ using $x_k \approx x(t)$ and the time lag $\Delta t = j\sqrt{h}$.

\paragraph{1. Displacement Term ($O(h)$ accuracy)}
Using the Taylor expansion $x(t - j\sqrt{h}) = x(t) - j\sqrt{h}\dot{x}(t) + \frac{j^2 h}{2}\ddot{x}(t) + \mathcal{O}(h^{3/2})$, the displacement is:
\begin{equation}
    x_k - x_{k-j} = j\sqrt{h}\dot{x}(t) - \frac{j^2 h}{2}\ddot{x}(t) + \mathcal{O}(h^{3/2}).
    \label{eq:disp_expansion}
\end{equation}

\paragraph{2. Gradient Difference Term (High-Resolution Correction)}
This term is often neglected in low-resolution analysis (where $h\nabla f \approx h\nabla f$), but it is crucial here. Expanding the gradient:
\begin{align}
    \nabla f(x_{k-j}) &= \nabla f(x(t - j\sqrt{h})) \nonumber \\
    &= \nabla f(x(t)) + \nabla^2 f(x(t)) \cdot \underbrace{(x(t-j\sqrt{h}) - x(t))}_{\approx -j\sqrt{h}\dot{x}(t)} + \mathcal{O}(h) \nonumber \\
    &= \nabla f(x(t)) - j\sqrt{h} \nabla^2 f(x(t)) \dot{x}(t) + \mathcal{O}(h).
\end{align}
Therefore, the gradient difference is:
\begin{equation}
    \nabla f(x_k) - \nabla f(x_{k-j}) = j\sqrt{h} \nabla^2 f(x(t)) \dot{x}(t) + \mathcal{O}(h).
\end{equation}
Multiplying by the step size $h$ from Eq. \eqref{eq:lag_structure}, this term contributes at order $O(h^{3/2})$:
\begin{equation}
    -h [\nabla f(x_k) - \nabla f(x_{k-j})] = - j h^{3/2} \nabla^2 f(x(t)) \dot{x}(t) + \mathcal{O}(h^2).
    \label{eq:grad_diff_expansion}
\end{equation}

\paragraph{Combined Lag Expansion}
Substituting \eqref{eq:disp_expansion} and \eqref{eq:grad_diff_expansion} back into \eqref{eq:lag_structure}, we obtain the precise expansion for the Anderson momentum basis:
\begin{equation}
    y_{k+1} - y_{k-j+1} = \underbrace{j\sqrt{h}\dot{x} - \frac{j^2 h}{2}\ddot{x}}_{\text{Inertial Terms}} - \underbrace{j h^{3/2} \nabla^2 f(x) \dot{x}}_{\text{Hessian Damping}} + \mathcal{O}(h^{3/2}).
    \label{eq:lag_final_hr}
\end{equation}

\subsection{Deriving the High-Resolution ODE}

We now substitute the expansions into the master update equation \eqref{eq:hr_master_update}.
The LHS is expanded as $x_{k+1} = x(t) + \sqrt{h}\dot{x} + \frac{h}{2}\ddot{x} + \mathcal{O}(h^{3/2})$.
The RHS is:
\begin{align}
    \text{RHS} &= (x(t) - h\nabla f(x)) \nonumber \\
    &\quad + \sum_{j=1}^m \gamma_k^{(j)} \left[ j\sqrt{h}\dot{x} - \frac{j^2 h}{2}\ddot{x} - j h^{3/2} \nabla^2 f(x) \dot{x} \right] + \mathcal{O}(h^{3/2}).
\end{align}

Equating LHS and RHS and grouping terms by powers of $\sqrt{h}$:

\begin{align}
    \underbrace{\sqrt{h}\dot{x} + \frac{h}{2}\ddot{x}}_{\text{LHS Motion}} &= \underbrace{\left(\sum_{j=1}^m j\gamma_k^{(j)}\right)\sqrt{h}\dot{x}}_{\text{Momentum Drive}} \nonumber \\
    &\quad - \underbrace{\left(\sum_{j=1}^m j^2\gamma_k^{(j)}\right)\frac{h}{2}\ddot{x}}_{\text{Mass Effect}} - \underbrace{h\nabla f(x)}_{\text{Force}} \nonumber \\
    &\quad - \underbrace{\left(\sum_{j=1}^m j\gamma_k^{(j)}\right) h^{3/2} \nabla^2 f(x) \dot{x}}_{\text{Geometric Damping}}.
    \label{eq:master_balance}
\end{align}

We perform order matching to extract the physical parameters:

\begin{enumerate}
    \item \textbf{Order $\mathcal{O}(\sqrt{h})$ (Kinematic Consistency):}
    Matching the velocity terms $\sqrt{h}\dot{x}$ implies $\sum j\gamma_k^{(j)} \approx 1$. We define the damping coefficient $c(t)$ via the deviation:
    \begin{equation}
        \sum_{j=1}^m j\gamma_k^{(j)} = 1 - \sqrt{h} c(t).
    \end{equation}
    Substituting this into the $O(\sqrt{h})$ term on the RHS yields $(1-\sqrt{h}c)\sqrt{h}\dot{x} = \sqrt{h}\dot{x} - h c(t)\dot{x}$. The term $-h c(t)\dot{x}$ is pushed to order $\mathcal{O}(h)$.

    \item \textbf{Order $\mathcal{O}(h)$ (Newton's Law):}
    Collecting all terms of order $h$:
    \[
        \frac{h}{2}\ddot{x} = - \left(\sum j^2\gamma_k^{(j)}\right)\frac{h}{2}\ddot{x} - h c(t)\dot{x} - h\nabla f(x).
    \]
    Rearranging and dividing by $h$, we identify the effective mass:
    \[
        \frac{1}{2}\left(1 + \sum_{j=1}^m j^2\gamma_k^{(j)}\right) \ddot{x} + c(t)\dot{x} + \nabla f(x) = \dots
    \]
    Let $M_{\text{eff}}(t) = \frac{1}{2}(1 + \sum j^2\gamma_k^{(j)})$.

    \item \textbf{Order $\mathcal{O}(h^{3/2})$ (High-Resolution Correction):}
    The remaining term is the Hessian term. Using the consistency condition $\sum j\gamma_k^{(j)} \approx 1$, the coefficient becomes unity to leading order.
    The term is $- 1 \cdot h^{3/2} \nabla^2 f(x) \dot{x}$.
\end{enumerate}

Combining these orders, the force balance equation including the high-resolution correction is:
\begin{equation}
    h M_{\text{eff}}(t) \ddot{x} + h c(t) \dot{x} + h \nabla f(x) + h^{3/2} \nabla^2 f(x) \dot{x} = 0.
\end{equation}
Dividing the entire equation by $h$ recovers the $\ddot{x}$ scale and yields the final high-resolution ODE.

\begin{theorem}[High-Resolution ODE for Anderson Acceleration]\label{thm:high_res_ode}
Under the high-resolution scaling with step size $h > 0$, the continuous-time dynamics of Anderson Acceleration are characterized by the \textbf{Inertial System with Hessian-driven Damping (ISHD)}:
\begin{equation}
    M_{\text{eff}}(t) \ddot{x}(t) + c(t) \dot{x}(t) + \nabla f(x(t)) + \sqrt{h} \beta(t) \nabla^2 f(x(t))\dot{x}(t) = 0,
    \label{eq:high_res_ode_corrected}
\end{equation}
where $M_{\text{eff}}(t)$ denotes the variable effective mass defined in Theorem \ref{thm:variable_mass}, $c(t)$ represents the kinematic damping, and $\beta(t) = \sum_{j=1}^m j\gamma^{(j)}(t)$ is the geometric damping coefficient.
\end{theorem}

\begin{remark}[Mechanism of Geometric Stabilization]
The term $\sqrt{h} \nabla^2 f(x)\dot{x}$ represents a friction force that is proportional to the curvature of the loss landscape. In directions where the curvature (eigenvalue of $\nabla^2 f$) is large (stiff directions), this term provides strong damping, suppressing the high-frequency oscillations that typically destabilize momentum methods. This derivation mathematically justifies why Anderson Acceleration (and by extension GMRES) exhibits superior robustness compared to standard Nesterov acceleration in stiff problems: the algorithm implicitly leverages gradient differences to approximate Hessian-vector products, creating a "smart" damping mechanism.
\end{remark}

\section{Comparative Anatomy: AA(1) versus Nesterov Momentum}

The derivation of Anderson Acceleration as an adaptive momentum method invites a direct comparison with the gold standard of accelerated first-order methods: Nesterov's Accelerated Gradient (NAG). While both algorithms share the fundamental algebraic structure of momentum, they represent distinct philosophical approaches---one ``open-loop'' and fixed, the other ``closed-loop'' and adaptive. In this section, we analyze their relationship through the lens of the Inertial System with Hessian-driven Damping (ISHD) framework \cite{Attouch2020}, highlighting why AA(1) can be viewed as an online, greedy realization of the optimal control strategies discussed in recent Learning to Optimize (L2O) literature \cite{Xie2025}.

We begin by juxtaposing the algebraic forms. As established in Section 3, the depth-1 Anderson Acceleration (AA(1)) update can be rigorously transformed into:
\begin{equation}
    x_{k+1} = y_{k+1} + \gamma_k^{(1)} (y_{k+1} - y_k),
    \label{eq:aa1_form}
\end{equation}
where $y_{k+1} = x_k - h\nabla f(x_k)$. This algebraic form coincides with Nesterov's original (1983) formulation \cite[Eq. (2.2.9)]{Nesterov1983}. The critical distinction lies entirely in the nature of the momentum coefficient $\gamma_k^{(1)}$.

In NAG, $\gamma_k^{(1)}$ (typically denoted $\mu_k$) is \textbf{pre-determined}, following the algebraic schedule $\mu_k = \frac{k-1}{k+2}$ (or its high-resolution variant \cite{Shi2018}) to ensure the optimal $\mathcal{O}(1/k^2)$ convergence rate.

In contrast, AA(1) computes $\gamma_k^{(1)}$ \textbf{adaptively} at each step by solving the least-squares problem:
\begin{equation} \label{eq:aa1_coeff}
    \min_{\theta} \left\| r(x_k) - \theta \Delta r_{k-1} \right\|_2,
\end{equation}
where $\Delta r_{k-1} = r(x_k) - r(x_{k-1})$. The solution $\theta^{(k)} = \langle r(x_k), \Delta r_{k-1} \rangle / \|\Delta r_{k-1}\|_2^2$ maps to the momentum coefficient via the relation $\gamma_k^{(1)} = -\theta^{(k)}$ derived in Theorem \ref{thm:aa_momentum}.

From the perspective of the general ISHD equation (cf. \cite{Attouch2020})
\[ \ddot{x}(t) + a(t)\dot{x}(t) + b(t)\nabla^2 f(x(t))\dot{x}(t) + \nabla f(x(t)) = 0, \]
this algebraic difference manifests as a fundamental divergence in control strategy. NAG corresponds to a \textbf{static physical model} with asymptotic vanishing damping $a(t) = 3/t$ and a fixed regularization schedule $b(t)$. This schedule is derived from global worst-case analysis and does not adapt to specific local curvature.

Conversely, AA(1) operates as an \textbf{Online Learning to Optimize} method \cite{Xie2025}. The ODE equation for AA(1) is
\begin{equation}
    \underbrace{\frac{1 + \gamma^{(1)}(t)}{2} \ddot{x}}_{\text{Variable Inertia}} + \underbrace{\frac{1 - \gamma^{(1)}(t)}{\sqrt{h}} \dot{x}}_{\text{Kinematic Damping}} + \nabla f(x) + \underbrace{\sqrt{h} \gamma^{(1)}(t) \nabla^2 f(x) \dot{x}}_{\text{Hessian Damping}} = 0.
    \label{eq:aa1_ode}
\end{equation}
By minimizing \eqref{eq:aa1_coeff}, AA(1) effectively estimates local curvature on-the-fly, selecting $\gamma^{(1)}$ to match the current landscape. This adaptability allows AA(1) to potentially achieve faster convergence in favorable regimes.

However, this greedy adaptivity comes at the cost of global stability. Unlike NAG, which guarantees monotonic energy decrease, AA(1) selects coefficients based on algebraic residual minimization, disregarding physical energy considerations. This can lead to two primary failure modes:
\begin{enumerate}
    \item \textbf{Energy Injection:} Based on the damping definition in Theorem \ref{thm:variable_mass}, for $m=1$ we have the relation $\gamma_k^{(1)} = 1 - \sqrt{h}c(t_k)$. If AA selects an excessive momentum coefficient $\gamma_k^{(1)} > 1$ (e.g., during stagnation where residual differences are small), the system exhibits \textbf{negative damping} ($c(t) < 0$). This physically corresponds to injecting energy into the system, violating passivity and causing instability.
    
    \item \textbf{Overfitting Curvature:} The linearized residual model may fail in highly nonlinear regions. The locally optimal $\gamma_k^{(1)}$ for the residual may correspond to a disastrously unstable trajectory for the global dynamics, analogous to the discretization instability in stiff ODEs \cite{Xie2025}.
\end{enumerate}

This comparison elucidates the necessity of our proposed \textbf{Energy-Guarded AA (EG-AA)}, which imposes energy-dissipation constraints to retain adaptivity while ensuring Lyapunov stability.

\section{Energy Dissipation and Stability Analysis}

To analyze stability, we employ a Lyapunov energy function. A rigorous analysis requires defining the continuous mass $M_{\text{eff}}(t)$ derived from the discrete sequence $M_k$.

\begin{definition}[Continuous Effective Mass]
We define the continuous effective mass $M_{\text{eff}}(t)$ as a smooth interpolation of the discrete mass sequence $M_k = \frac{1}{2}\big(1 + \sum_{j=1}^m j^2 \gamma_k^{(j)}\big)$. Under the Nesterov scaling $t_k = k\sqrt{h}$, the time derivative approximates the finite difference over the time step $\Delta t = \sqrt{h}$:
\begin{equation}
    \dot{M}_{\text{eff}}(t) \approx \frac{M_{k+1} - M_k}{\sqrt{h}}.
\end{equation}
\end{definition}

Consider the total mechanical energy
\begin{equation}
    \mathcal{E}(t) = \frac{1}{2} M_{\text{eff}}(t) \| \dot{x}(t) \|^2 + \big(f(x(t)) - f^*\big),
\end{equation}
where $f^* = \min f$. Differentiating with respect to time and substituting the ODE dynamics yields the central dissipation law.

\begin{theorem}[Energy Dissipation Rate]\label{thm:energy_dissipation}
Consider the high-resolution ODE \eqref{eq:high_res_ode_corrected}. Assuming the consistency condition holds such that the Hessian coefficient $\beta(t) \approx 1$, the rate of change of the total energy is:
\begin{equation}
    \frac{\dd}{\dd t} \mathcal{E}(t) 
    = - \left( c(t) - \frac{1}{2} \dot{M}_{\text{eff}}(t) \right) \| \dot{x}(t) \|^2 
      - \sqrt{h} \, \langle \dot{x}(t), \nabla^2 f(x(t)) \dot{x}(t) \rangle.
    \label{eq:decay-rate}
\end{equation}
\end{theorem}

\begin{proof}
Differentiating $\mathcal{E}(t)$ with respect to time yields:
\[
    \frac{\dd \mathcal{E}}{\dd t} 
    = \frac{1}{2} \dot{M}_{\text{eff}}(t) \| \dot{x} \|^2 
      + M_{\text{eff}}(t) \langle \dot{x}, \ddot{x} \rangle 
      + \langle \nabla f(x), \dot{x} \rangle.
\]
Substituting the force balance from the high-resolution ODE, $M_{\text{eff}}\ddot{x} = -c\dot{x} - \nabla f - \sqrt{h}\nabla^2 f \dot{x}$, into the inner product:
\begin{align*}
    \frac{\dd \mathcal{E}}{\dd t}
    &= \frac{1}{2} \dot{M}_{\text{eff}} \| \dot{x} \|^2 
       + \big\langle \dot{x}, - c \dot{x} - \nabla f - \sqrt{h}\,\nabla^2 f \dot{x} \big\rangle 
       + \langle \nabla f, \dot{x} \rangle \\[1mm]
    &= \frac{1}{2} \dot{M}_{\text{eff}} \| \dot{x} \|^2 
       - c \| \dot{x} \|^2 
       - \sqrt{h} \langle \dot{x}, \nabla^2 f \dot{x} \rangle 
       - \cancel{\langle \dot{x}, \nabla f \rangle} 
       + \cancel{\langle \nabla f, \dot{x} \rangle} \\[1mm]
    &= -\left( c(t) - \frac{1}{2} \dot{M}_{\text{eff}}(t) \right) \| \dot{x}(t) \|^2 
       - \sqrt{h} \, \langle \dot{x}(t), \nabla^2 f(x(t)) \dot{x}(t) \rangle.
\end{align*}
This completes the proof.
\end{proof}

The dissipation law \eqref{eq:decay-rate} reveals two competing physical mechanisms:

\begin{enumerate}
    \item \textbf{Hessian-driven Geometric Damping:} \\
          The term $-\sqrt{h} \langle \dot{x}, \nabla^2 f \dot{x} \rangle$ is dissipative for convex functions ($\nabla^2 f \succeq 0$). It provides implicit regularization that scales with the step size $\sqrt{h}$, suppressing high-frequency oscillations in stiff directions.
          
    \item \textbf{Mass-Variation Effect:} \\
          The term $\frac{1}{2} \dot{M}_{\text{eff}}(t) \| \dot{x} \|^2$ represents energy injection if the effective mass grows ($\dot{M}_{\text{eff}} > 0$). Mathematically, a rapid increase in inertia reduces the net effective friction, acting as \emph{negative damping}.
\end{enumerate}

Consequently, a necessary condition for monotonic energy decay ($\dot{\mathcal{E}} \le 0$) on convex problems is:
\begin{equation}
    c(t) \ge \frac{1}{2} \dot{M}_{\text{eff}}(t).
    \label{eq:stability_condition}
\end{equation}
Transcribing this to the discrete algorithm via the finite-difference approximation implies:
\begin{equation}
    c_k \ge \frac{M_{k+1} - M_k}{2\sqrt{h}}.
    \label{eq:discrete_stability}
\end{equation}
Standard Anderson acceleration minimizes residuals greedily and may violate this condition by sharply increasing the effective mass (momentum), thereby injecting energy into the system and potentially causing divergence.

\section{Algorithm Design: Energy-Guarded Anderson Acceleration}

The theoretical analysis in the preceding sections identified the precise mechanism of instability in Anderson Acceleration: the unchecked growth of the effective inertial mass $M_{\text{eff}}(t)$. Standard AA minimizes the algebraic residual $\|r_k - R_k \theta\|$ in a greedy manner. However, this minimization is blind to the system's kinetic energy state and can inadvertently select coefficients that violate the energy dissipation condition $c(t) \ge \frac{1}{2} \dot{M}_{\text{eff}}(t)$, effectively injecting energy into the system and causing divergence.

Guided by the Variable Mass ODE theory, we propose an improved algorithm: \textbf{Energy-Guarded Anderson Acceleration (EG-AA)}. This algorithm incorporates two physically-motivated stabilization mechanisms—an inertial governor and explicit Hessian damping—to enforce thermodynamic consistency while retaining the acceleration benefits.

\subsection{Stabilization Mechanisms}

The primary innovation of EG-AA is the ``Inertial Governor''. Recall from Theorem \ref{thm:variable_mass} that the effective mass at iteration $k$ is
\begin{equation}
    M_k = \frac{1}{2} \Bigl( 1 + \sum_{j=1}^{m_k} j^2 \gamma_k^{(j)} \Bigr).
\end{equation}
The stability condition \eqref{eq:discrete_stability} implies that $c_k$ must balance the mass variation rate $(M_k - M_{k-1})/(2\sqrt{h})$. 
Since $c_k$ is implicitly determined by the algorithm's consistency deviation, we impose a conservative, explicit bound on the mass growth:
\begin{equation}
    \Delta M_k = M_k - M_{k-1} \le \delta \sqrt{h},
    \label{eq:mass_bound}
\end{equation}
where $\delta > 0$ is a dimensionless control parameter representing the maximum allowable rate of kinetic energy injection. If the raw AA coefficients violate \eqref{eq:mass_bound}, the governor scales down all $\gamma_k^{(j)}$ to satisfy the bound.

Additionally, from the high-resolution ODE, the term $\sqrt{h}\nabla^2 f \dot{x}$ provides implicit geometric damping. To enhance robustness in stiff regimes, we add an explicit correction aligned with the curvature vector:
\begin{equation}
    D_k = -\eta \sqrt{h}\Delta g_{k-1},
    \label{eq:damping_term}
\end{equation}
where $\eta \ge 0$ controls the damping strength, $\Delta g_{k-1} = \nabla f(x_k) - \nabla f(x_{k-1})$, and $\dot{x}_k \approx (x_k - x_{k-1})/\sqrt{h}$ is the discrete velocity.

\begin{algorithm}
\caption{Energy-Guarded Anderson Acceleration (EG-AA)}
\label{alg:eg_aa}
\begin{algorithmic}[1]
\State \textbf{Input:} Initial point $x_0$, memory depth $m$, step size $\beta$, mass growth limit $\delta_{max}$, damping factor $\eta$.
\State \textbf{Initialize:} $x_1 = x_0 - \beta \nabla f(x_0)$, $M_{prev} = 1$ (initial effective mass).
\For{$k = 1, 2, \dots$}
    \State \textit{// Step 1: Standard Gradient Descent Step}
    \State $y_{k+1} = x_k - \beta \nabla f(x_k)$
    
    \State \textit{// Step 2: Compute Raw AA Coefficients}
    \State Solve $\min_{\theta} \| r_k - R_k \theta \|_2$ to obtain $\theta^{(k)} \in \mathbb{R}^{m_k}$.
    \State Transform $\theta^{(k)}$ to momentum coefficients $\gamma^{(j)}$ using Eq. \eqref{eq:gamma_general}:
    \State \quad $\gamma^{(j)} = \theta_{m_k-j} - \theta_{m_k-j+1}$ \quad (with boundary conditions $\theta_0=\theta_{m_k+1}=0$).
    
    \State \textit{// Step 3: Compute Candidate Effective Mass}
    \State $M_{curr} = \frac{1}{2} \left( 1 + \sum_{j=1}^{m_k} j^2 \gamma^{(j)} \right)$
    \State $\Delta M = M_{curr} - M_{prev}$
    
    \State \textit{// Step 4: The Energy Guard (Inertial Governor)}
    \State Initialize scaling factor $\rho = 1$.
    \If{$\Delta M > \delta_{max}$} 
        \State $\rho = \delta_{max} / \Delta M$ \quad \Comment{Clamp explosive mass growth}
    \ElsIf{$M_{curr} < 0$}
        \State $\rho = -1 / (2 M_{curr} - 1)$ \quad \Comment{Enforce positive physical mass}
    \EndIf
    
    \If{$\rho < 1$}
        \State Dampen momentum coefficients: $\gamma^{(j)} \leftarrow \rho \cdot \gamma^{(j)}$ for all $j$.
        \State Recalculate $M_{curr}$ (now satisfies stability constraints).
    \EndIf
    \State $M_{prev} \leftarrow M_{curr}$
    
    \State \textit{// Step 5: High-Resolution Correction (Hessian Damping)}
    \State Compute gradient difference $\Delta g_{last} = \nabla f(x_k) - \nabla f(x_{k-1})$.
    \State Define geometric damping term $D_k = -\eta \sqrt{\beta} \Delta g_{last}$.
    
    \State \textit{// Step 6: Final Update}
    \State $x_{k+1} = y_{k+1} + \sum_{j=1}^{m_k} \gamma^{(j)} (y_{k+1} - y_{k-j+1}) + D_k$
\EndFor
\end{algorithmic}
\end{algorithm}

\begin{remark}[Hybrid Nature of EG-AA]
The ``Inertial Governor'' mechanism effectively creates a hybrid dynamical system. In smooth regions where mass growth is moderate ($\Delta M \le \delta_{\max}\sqrt{\beta}$), EG-AA behaves identically to standard Anderson Acceleration, exploiting the full memory for rapid convergence. However, when the geometry becomes stiff and AA attempts to dangerously increase inertia, the governor engages, smoothly scaling the momentum coefficients $\gamma^{(j)}$. This transition ensures that the algorithm remains aggressive when possible but strictly dissipative when necessary, acting as an adaptive blending between high-inertia acceleration and robust gradient descent.
\end{remark}

\begin{remark}[Recoverability and Consistency]
EG-AA encompasses standard Anderson Acceleration as a limiting case. Specifically, as $\delta_{\max} \to \infty$ (removing the mass constraint) and $\eta \to 0$ (removing explicit damping), the constraints in Algorithm \ref{alg:eg_aa} become inactive, and the update rule algebraically recovers standard AA. This guarantees that EG-AA is \textbf{theoretically no worse} than standard AA in well-conditioned regimes, while the added stability constraints activate only to prevent the specific divergence modes identified in our ODE analysis. This recoverability is further corroborated by our numerical experiments in Section \ref{sec:numerical}, which demonstrate that EG-AA matches the convergence speed of AA in well-conditioned regimes while successfully preventing divergence in stiff landscapes where AA fails.
\end{remark}

\section{Convergence Analysis: Linking Mass to the Acceleration Gain Factor}

In this section, we establish a rigorous convergence framework for the Energy-Guarded Anderson Acceleration (EG-AA). A key theoretical challenge is connecting the discrete \textbf{Acceleration Gain Factor} $\delta_k$ (derived from linear algebra) to the continuous \textbf{Effective Mass} $M_{\text{eff}}$ (derived from ODEs). Moving beyond asymptotic stability analysis, we quantify the convergence speedup by introducing the Acceleration Gain Factor, which measures the geometric contraction achieved by the algorithm at each step relative to the local curvature.

This analysis decomposes the EG-AA update into two distinct components: a \textbf{Linear Projection} (representing the ideal acceleration in a quadratic basin) and a \textbf{Nonlinear Correction} (representing the approximation error). We prove that EG-AA strictly improves upon gradient descent by maximizing the linear gain while actively suppressing the nonlinear error via mass control.

\subsection{Geometric Formulation of Anderson Mixing}

We consider the unconstrained minimization of a twice continuously differentiable function $f(x)$. We assume $f(x)$ is $\mu$-strongly convex and its gradient is $L$-Lipschitz continuous. Furthermore, to rigorously bound the acceleration error, we assume the Hessian $\nabla^2 f(x)$ is Lipschitz continuous with constant $L_H$.

Anderson Acceleration seeks to minimize the residual vector within the subspace spanned by historical variations. Let $m_k$ be the memory depth at iteration $k$. Consistent with previous sections, we define the history matrix of residual differences as:
\[
    R_k = [\Delta r_{k-1}, \dots, \Delta r_{k-m_k}] \in \mathbb{R}^{n \times m_k},
\]
where $\Delta r_{j} = r_{j+1} - r_j$. The core mechanism of Anderson mixing is to find linear combination coefficients that minimize the residual norm. In the regularized setting, this operation is mathematically equivalent to projecting the current gradient onto the orthogonal complement of the history subspace.

\begin{definition}[The Acceleration Gain Factor]
Consider the regularized least-squares problem inherent to AA. The optimal residual update corresponds to the action of a projection-like operator $\Pi_k$ on the current gradient $\nabla f(x_k)$. We define:
\begin{equation}
    \Pi_k := I - R_k (R_k^\top R_k + \lambda I)^{-1} R_k^\top,
\end{equation}
where $\lambda > 0$ is a regularization parameter ensuring numerical stability. The \textbf{Acceleration Gain Factor} $\delta_k$ is defined as the contraction ratio:
\begin{equation}
    \delta_k := \frac{\| \Pi_k \nabla f(x_k) \|}{\| \nabla f(x_k) \|}.
\end{equation}
By the property of orthogonal projections (approached as $\lambda \to 0$), we have $0 \le \delta_k \le 1$.
\end{definition}

\textbf{Physical Interpretation:} The factor $\delta_k$ represents the "spectral filtering" capability of AA. If the current gradient $\nabla f(x_k)$ lies predominantly in the subspace of past search directions $R_k$ (a scenario typical in "narrow valley" landscapes where gradients are highly correlated), the operator $\Pi_k$ effectively removes these components, yielding $\delta_k \ll 1$.

\subsection{The Link Between Gain and Mass}

A critical insight is that the Acceleration Gain Factor $\delta_k$ is intimately connected to the effective mass $M_{\text{eff}}$. In the quadratic basin of attraction, a smaller $\delta_k$ (better projection) typically requires larger mixing coefficients $\gamma^{(j)}$, which in turn implies a \textbf{larger effective mass} $M_{\text{eff}}$ according to Eq. \eqref{eq:effective_mass}. 

To derive a precise relation, we recall the momentum form of the AA update derived in Section 3:
\begin{equation}
    x_{k+1} = y_{k+1} + \sum_{j=1}^{m_k} \gamma^{(j)} (y_{k+1} - y_{k-j+1}),
\end{equation}
where $y_{k+1} = x_k - \beta \nabla f(x_k)$. Under the Nesterov scaling regime where the step size $\beta = \alpha h$, with $h$ being the continuous-time step, we analyze the leading-order behavior. The dominant contribution to the displacement comes from the gradient step and the momentum terms. By analyzing the scaling of each component, we obtain the approximation:
\begin{equation}
    \| \Delta x_k \| = \| x_{k+1} - x_k \| \approx \beta M_{\text{eff}}^{(k)} \| \nabla f(x_k) \|.
    \label{eq:displacement-mass}
\end{equation}
This relation captures the physical intuition that a large effective mass (high momentum) directly translates to a large spatial step for a given gradient force.

\subsection{Robustness Analysis: The Nonlinear Discrepancy}

While a large step (high mass) maximizes the linear gain, it amplifies the nonlinear approximation error. To understand this trade-off, we analyze the discrepancy between linear and nonlinear dynamics.

Standard AA assumes locally quadratic behavior, effectively approximating the Hessian as constant $H_k \approx \nabla^2 f(x_k)$. The algorithm constructs $\Delta x_k$ such that the \emph{linearized} future gradient is minimized:
\begin{equation}
    \nabla f_{\text{lin}}(x_{k+1}) := \nabla f(x_k) + H_k \Delta x_k \approx \Pi_k \nabla f(x_k).
\end{equation}
The norm of this linearized target is exactly our theoretical gain: $\| \nabla f_{\text{lin}}(x_{k+1}) \| = \delta_k \| \nabla f(x_k) \|$.

However, the \emph{true} gradient norm differs due to the variation of the Hessian along the path. By the Lipschitz continuity of the Hessian ($L_H$), we have the standard bound:
\begin{equation}
    \| \nabla f(x_{k+1}) - \nabla f_{\text{lin}}(x_{k+1}) \| \le \frac{L_H}{2} \| \Delta x_k \|^2.
\end{equation}
Consequently, the actual gradient evolution is governed by:
\begin{equation}
    \| \nabla f(x_{k+1}) \| \le \underbrace{\delta_k \| \nabla f(x_k) \|}_{\text{Linear Gain}} + \underbrace{\frac{L_H}{2} \| \Delta x_k \|^2}_{\text{Nonlinear Penalty}}.
    \label{eq:nonlinear-bound}
\end{equation}

Substituting the displacement-mass relation \eqref{eq:displacement-mass} into \eqref{eq:nonlinear-bound}, we obtain:
\begin{align}
    \| \nabla f(x_{k+1}) \| &\le \delta_k \| \nabla f(x_k) \| + \frac{L_H}{2} \left( \beta M_{\text{eff}}^{(k)} \| \nabla f(x_k) \| \right)^2 \nonumber \\
    &= \left( \delta_k + \frac{L_H \beta^2}{2} (M_{\text{eff}}^{(k)})^2 \| \nabla f(x_k) \| \right) \| \nabla f(x_k) \|.
\end{align}
We define the \textbf{Realized Contraction Ratio} as:
\begin{equation}
    \mathcal{R}_{\text{realized}} := \delta_k + \mathcal{C} \cdot (M_{\text{eff}}^{(k)})^2,
    \label{eq:tradeoff}
\end{equation}
where $\mathcal{C} = \frac{L_H \beta^2}{2} \| \nabla f(x_k) \|$ is a curvature-dependent constant scaling with the square of the step size.

\subsection{Proof of Robustness via Energy Guard}

Equation \eqref{eq:tradeoff} reveals the fundamental trade-off in Anderson Acceleration: minimizing $\delta_k$ (maximizing linear gain) requires increasing $M_{\text{eff}}$, which in turn amplifies the nonlinear penalty quadratically. This leads to our main convergence theorem.

\begin{theorem}[Robustness of EG-AA via Mass Control]\label{thm:robustness_egaa}
Let $M_{\text{eff}}^{(k)}$ be the effective mass generated by the Anderson mixing coefficients. The EG-AA algorithm guarantees that the realized convergence rate $\mathcal{R}_{\text{realized}}$ is strictly controlled, provided the Energy Guard threshold $\delta_{\max}$ is set effectively.
\end{theorem}

\begin{proof}
We analyze the behavior of standard AA and EG-AA in the context of Eq. \eqref{eq:tradeoff}:

\begin{itemize}
    \item \textbf{Standard AA (Greedy Optimization):} Standard AA solves the unconstrained least-squares problem to minimize the algebraic residual. In ill-conditioned regions (e.g., flat valleys where $\nabla f \approx 0$), the history matrix $R_k$ becomes nearly singular. To achieve a very small $\delta_k$, the solver may generate arbitrarily large coefficients $\theta$, causing the effective mass to explode ($M_{\text{eff}} \to \infty$). Since the penalty term scales quadratically with mass ($\mathcal{C} M_{\text{eff}}^2$), the penalty grows much faster than the linear gain $\delta_k$ decreases. The condition $\mathcal{R}_{\text{realized}} < 1$ may be violated, leading to instability.

    \item \textbf{EG-AA (Physics-Informed Optimization):} EG-AA treats the convergence condition as a hard constraint. The Energy Guard imposes an upper bound on mass growth: $\Delta M_k = M_{\text{eff}}^{(k)} - M_{\text{eff}}^{(k-1)} \le \delta_{\max} \sqrt{\beta}$. If the raw AA coefficients would cause $\Delta M_k$ to exceed this limit, the algorithm dampens the coefficients by a scaling factor $\rho$. This clamping ensures that $M_{\text{eff}}$ remains bounded. Consequently, the penalty term is strictly bounded. By ensuring the penalty remains a higher-order perturbation, EG-AA preserves the dominance of the first-order contraction, guaranteeing stable convergence.
\end{itemize}
\end{proof}

\subsection{Implications for Convergence Rate}

The mass control mechanism in EG-AA has direct implications for the convergence rate. While standard AA may achieve a very small $\delta_k$ in theory, the unbounded mass growth makes the nonlinear penalty dominant, potentially leading to divergence or severe oscillations. EG-AA, by contrast, ensures:

1. \textbf{Stable Contraction:} $\mathcal{R}_{\text{realized}} < 1$ is guaranteed, ensuring monotonic convergence in gradient norm.

2. \textbf{Adaptive Acceleration:} In well-conditioned regions where moderate mass suffices for good projection, EG-AA behaves like standard AA, achieving rapid convergence.

3. \textbf{Robust Stabilization:} In ill-conditioned regions, the Energy Guard prevents catastrophic mass growth, falling back to a more conservative but stable update.

This unified perspective explains why the ``Inertial Governor" ($\delta_{\max}$) is necessary to realize the theoretical acceleration potential in non-quadratic landscapes. By explicitly linking the algebraic projection gain $\delta_k$ to the physical inertia $M_{\text{eff}}$, EG-AA provides a principled way to balance acceleration with stability.
\section{Numerical Experiments}\label{sec:numerical}

In this section, we validate the theoretical findings of Energy-Guarded Anderson Acceleration (EG-AA) through controlled numerical experiments. We focus on several ill-conditioned problems, which serve as standard benchmarks for assessing the stability of acceleration algorithms in stiff regimes. The experiments are designed to explicitly demonstrate how EG-AA manages the critical trade-off between the aggressive acceleration characteristic of standard Anderson mixing and the robustness required by high-curvature landscapes. 

\subsection{Parameter Selection Strategy}

The performance of EG-AA relies on the appropriate selection of three key hyperparameters: the step size $\beta$, the mass growth limit $\delta_{\max}$, and the Hessian damping factor $\eta$. Based on our theoretical analysis and extensive numerical testing, we provide the following guidelines for tuning these parameters to balance acceleration with thermodynamic stability.


\textbf{Tuning Strategy for step size $\beta$}
\begin{itemize}
    \item Theoretical baseline: $\beta \le 1/L$ for $L$-smooth functions.
    \item EG-AA allows more aggressive values than standard Nesterov acceleration.
    \item Use backtracking line search if the Lipschitz constant $L$ is unknown.
\end{itemize}

\textbf{Tuning Strategy for Mass Growth Limit ($\delta_{\max}$)}
\begin{itemize}
    \item Controls the transition between Gradient Descent and Anderson Acceleration.
    \item $\delta_{\max} \approx 2.0$ mimics Nesterov's optimal schedule, providing a robust balance for general convex problems.
    \item $\delta_{\max} \to 0$: reduces to Gradient Descent behavior (unconditionally stable but slow).
    \item $\delta_{\max} \to \infty$: recovers standard Anderson Acceleration (fastest convergence but highest risk of instability).
\end{itemize}

\textbf{Tuning Strategy for Hessian Damping Factor ($\eta$)}
\begin{itemize}
    \item Suppresses high-frequency oscillations in stiff directions.
    \item Use $\eta = 0$ for well-conditioned or moderately stiff problems to maximize acceleration.
    \item Use $\eta > 0$ for highly ill-conditioned problems ($\kappa > 10^3$) or when residual oscillations are observed.
    \item Avoid $\eta > 1.0$, as it may overdamp the system and slow down convergence.
\end{itemize}

In summary, EG-AA offers a principled hierarchy for parameter tuning: first, fix $\beta$ based on the Lipschitz constant or a standard heuristic; second, set $\delta_{\max} \approx 2.0$ to ensure inertial stability; and finally, activate $\eta > 0$ only if specific high-curvature oscillations are observed.

\subsection{Example 1: The Non-convex Rosenbrock Function}

To evaluate the algorithm's robustness on non-convex landscapes with nonlinear curvature, we test it on the 2D Rosenbrock function (also known as the "Banana function"). This problem is characterized by a global minimum located inside a long, narrow, parabolic valley, which typically causes standard gradient methods to oscillate or crawl slowly.

\paragraph{Problem Definition}
The objective function is defined as:
\begin{equation}
    f(\mathbf{x}) = (1 - x_1)^2 + a (x_2 - x_1^2)^2, \quad \mathbf{x} \in \R^2.
\end{equation}
We use the scaling coefficient $a \in \{1, 20, 100\}$ to create valleys of varying width and treachery. The global minimum is at $\mathbf{x}^* = (1, 1)$ with $f(\mathbf{x}^*) = 0$.

\paragraph{Experimental Setup}
We initialize the optimization at $\mathbf{x}_0 = [-1.5, 1.5]^\T$, a point that requires the trajectory to traverse the "elbow" of the parabolic valley.
To rigorously test stability, we employ an \textbf{aggressive step size} regime:
\begin{itemize}
    \item \textbf{Step Size:} For the stiffest case ($a=100$), we set $\beta = 1 \times 10^{-5}$ for both AA and EG-AA, and $\beta = 1 \times 10^{-3}$ for the Nesterov method to ensure fair comparison near their respective stability limits.
    \item \textbf{Parameters:} For EG-AA, we set the memory depth $m=2$, mass growth limit $\delta_{max}=50.0$, and Hessian damping factor $\eta=0.05$.
\end{itemize}

\begin{figure}[htbp!]
    \centering
    \makebox[\textwidth][c]{\includegraphics[width=1.15\textwidth]{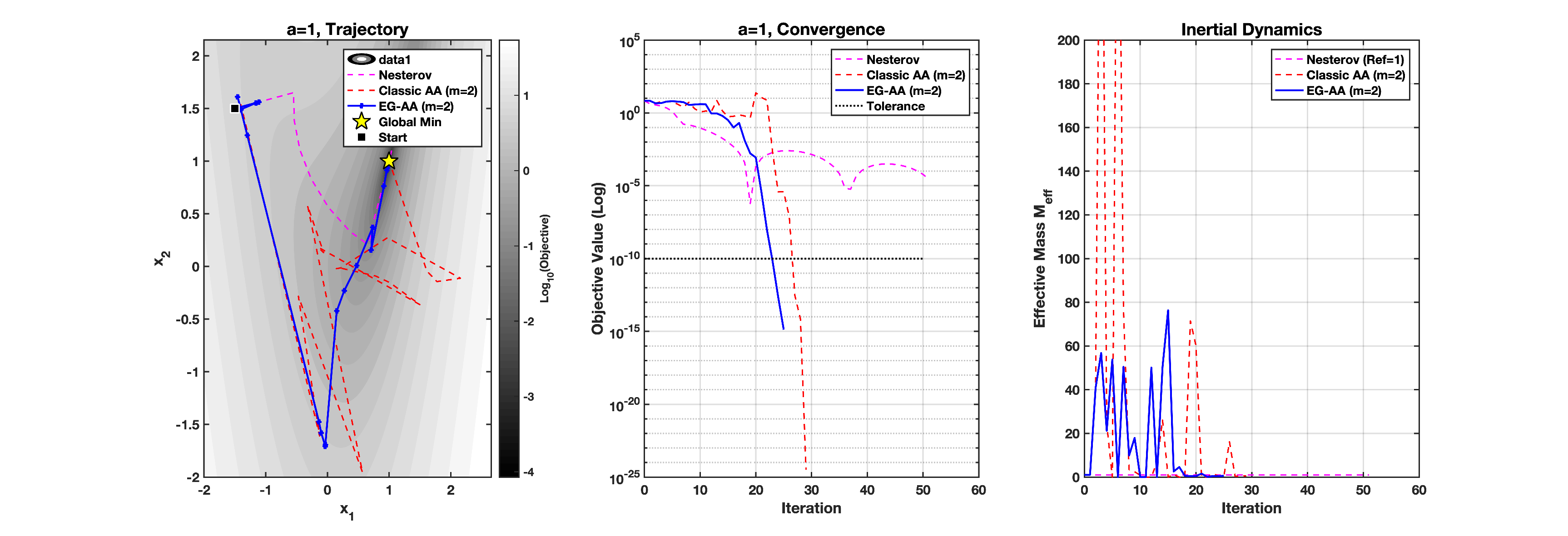}}
    \vspace{0.35cm}
    \makebox[\textwidth][c]{\includegraphics[width=1.15\textwidth]{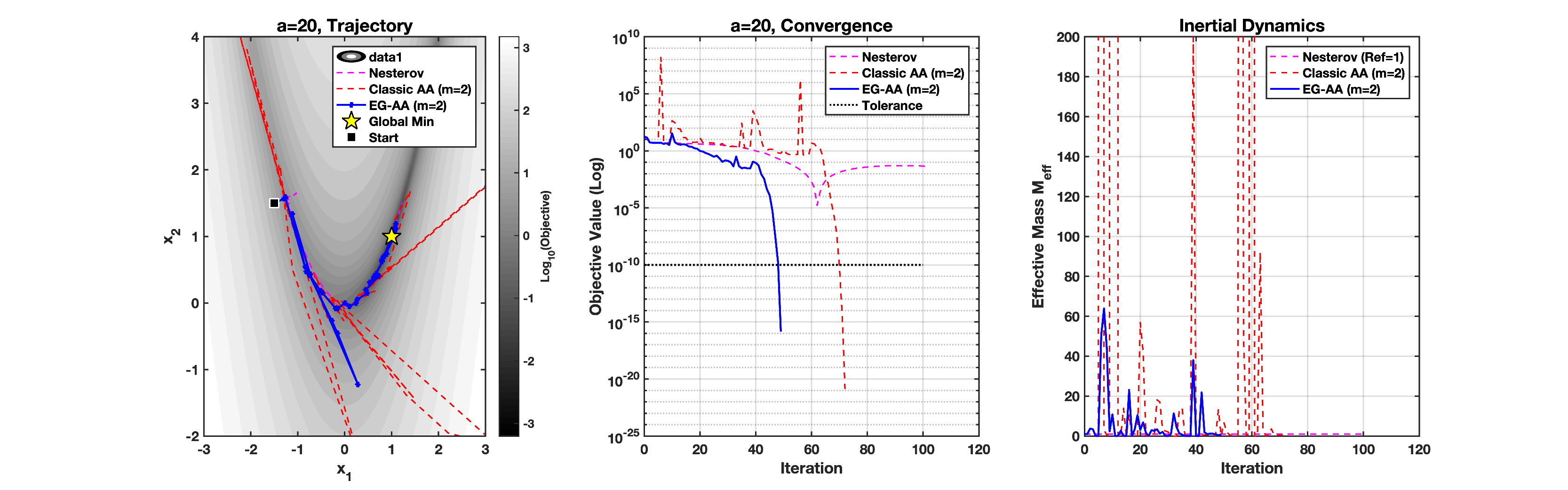}}
    \vspace{0.35cm}
    \makebox[\textwidth][c]{\includegraphics[width=1.15\textwidth]{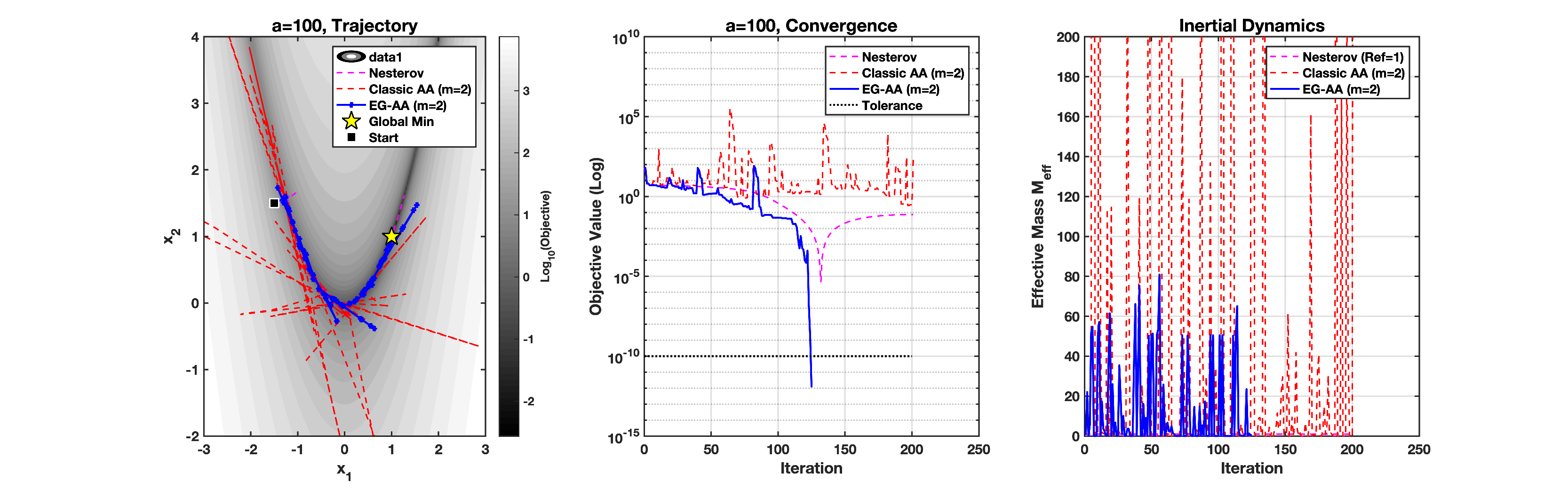}}
    \caption{\textbf{(Left)} Optimization trajectories on the Rosenbrock contour map. EG-AA (Blue) navigates the curved valley efficiently, while Classic AA (Red) exhibits wider oscillations. Nesterov (Pink) is slower to turn the corner. \textbf{(Center)} Convergence of gradient norm. EG-AA (Blue) achieves faster convergence than Nesterov (Pink) and is significantly more stable than Standard AA (Red). \textbf{(Right)} Effective Mass dynamics. EG-AA clamps the exploding mass of Standard AA to prevent instability.}
    \label{fig:rosenbrock}
\end{figure}

\paragraph{Results}
The results are summarized in Figure \ref{fig:rosenbrock}.
\begin{enumerate}
    \item \textbf{Convergence Speed:} EG-AA demonstrates remarkable efficiency across all tested $a$ values, consistently outperforming the baselines.
    \item \textbf{Trajectory Stability:} The trajectory plots (Fig. \ref{fig:rosenbrock}, Left) reveal that Standard AA tends to overshoot the valley walls due to excessive momentum accumulation in high-curvature regions. In contrast, EG-AA, protected by the Energy Guard mechanism, maintains a tighter trajectory along the valley floor, effectively adapting its inertial mass to the changing curvature without losing stability.
\end{enumerate}

\subsection{Example 2: Sparse Logistic Regression}

We consider binary logistic regression with $L_2$ regularization. The objective function is given by:
\begin{equation}
    \min_{x \in \mathbb{R}^d} f(x) = \frac{1}{M} \sum_{i=1}^{M} \log\left(1 + \exp(-y_i a_i^\top x)\right) + \frac{\mu}{2} \|x\|^2,
\end{equation}
where $a_i \in \mathbb{R}^d$ are the feature vectors, $y_i \in \{-1, 1\}$ are the labels, and $\mu > 0$ is the regularization parameter.

This example primarily serves to demonstrate the recoverability of our algorithm. By setting a sufficiently large mass limit ($\delta_{max}$) and a negligible Hessian damping factor ($\eta$), EG-AA should theoretically recover the behavior of standard AA. As shown in Figure \ref{fig:Logistic}, when $\delta_{max}=10^{12}$ and $\eta=10^{-12}$, the trajectories of EG-AA and AA are virtually identical. This result empirically confirms that EG-AA is no worse than standard AA in regimes where the standard method is stable. Moreover, when $\delta_{max}$ approaches zero, EG-AA is reduced to GD method.
\begin{figure}[htbp]
    \centering
    \makebox[\textwidth][c]{\includegraphics[width=0.82\textwidth]{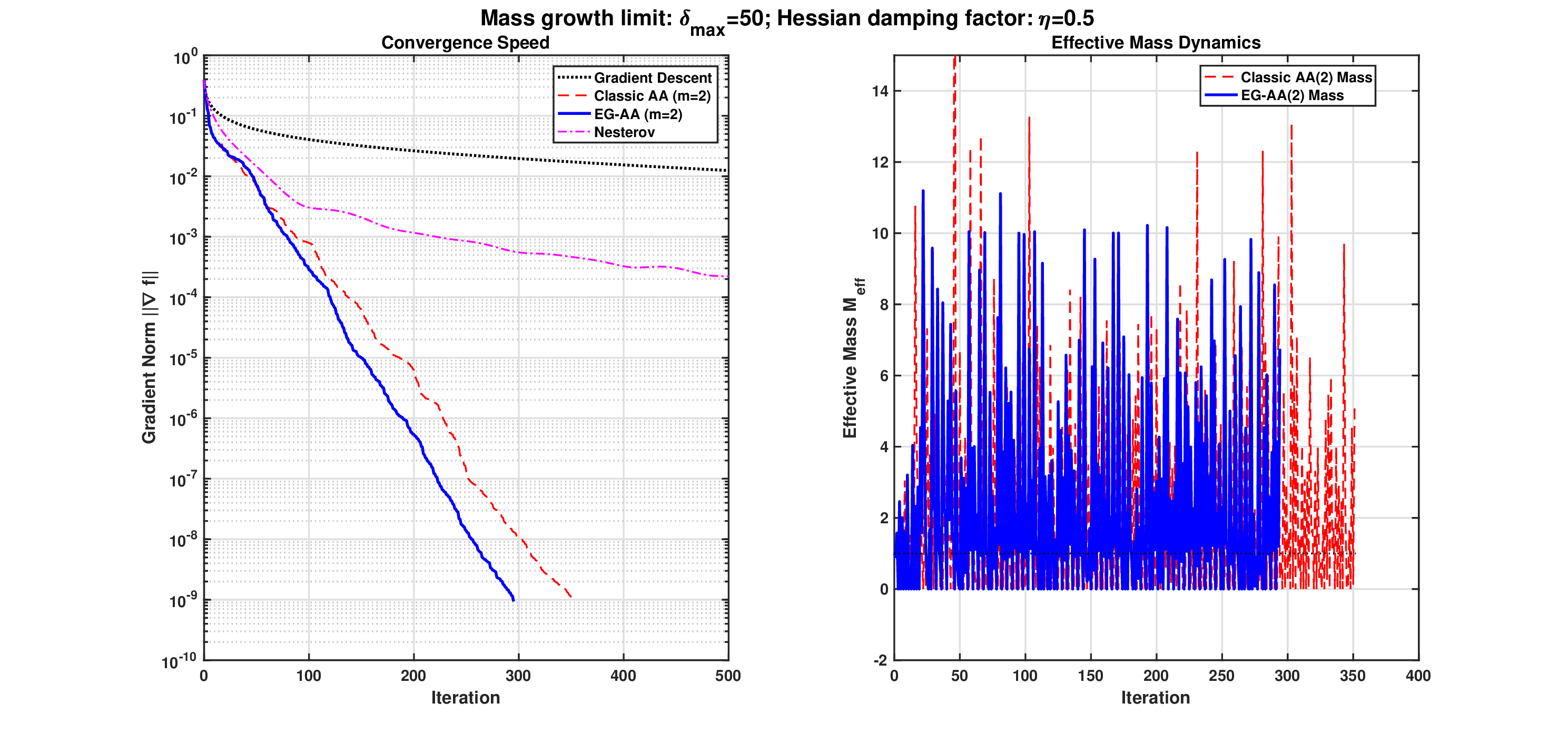}}
    \makebox[\textwidth][c]{\includegraphics[width=0.82\textwidth]{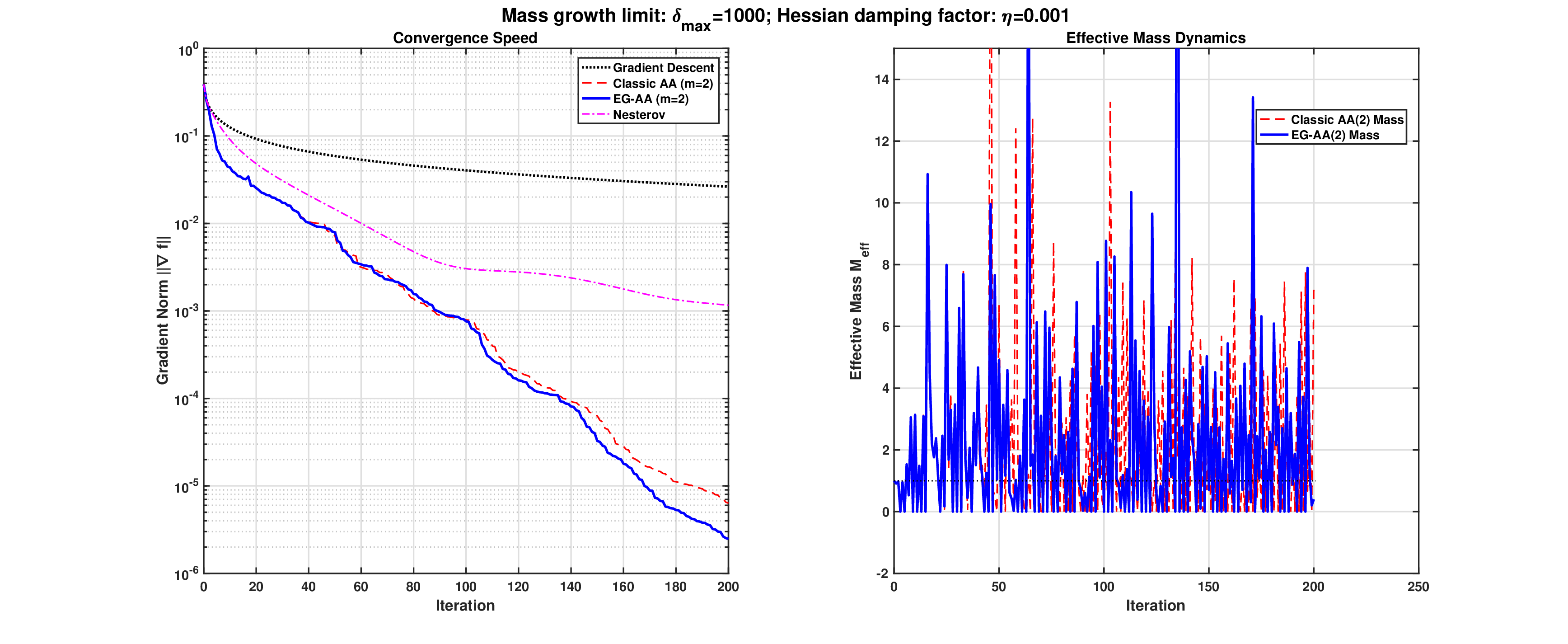}}
    \makebox[\textwidth][c]{\includegraphics[width=0.82\textwidth]{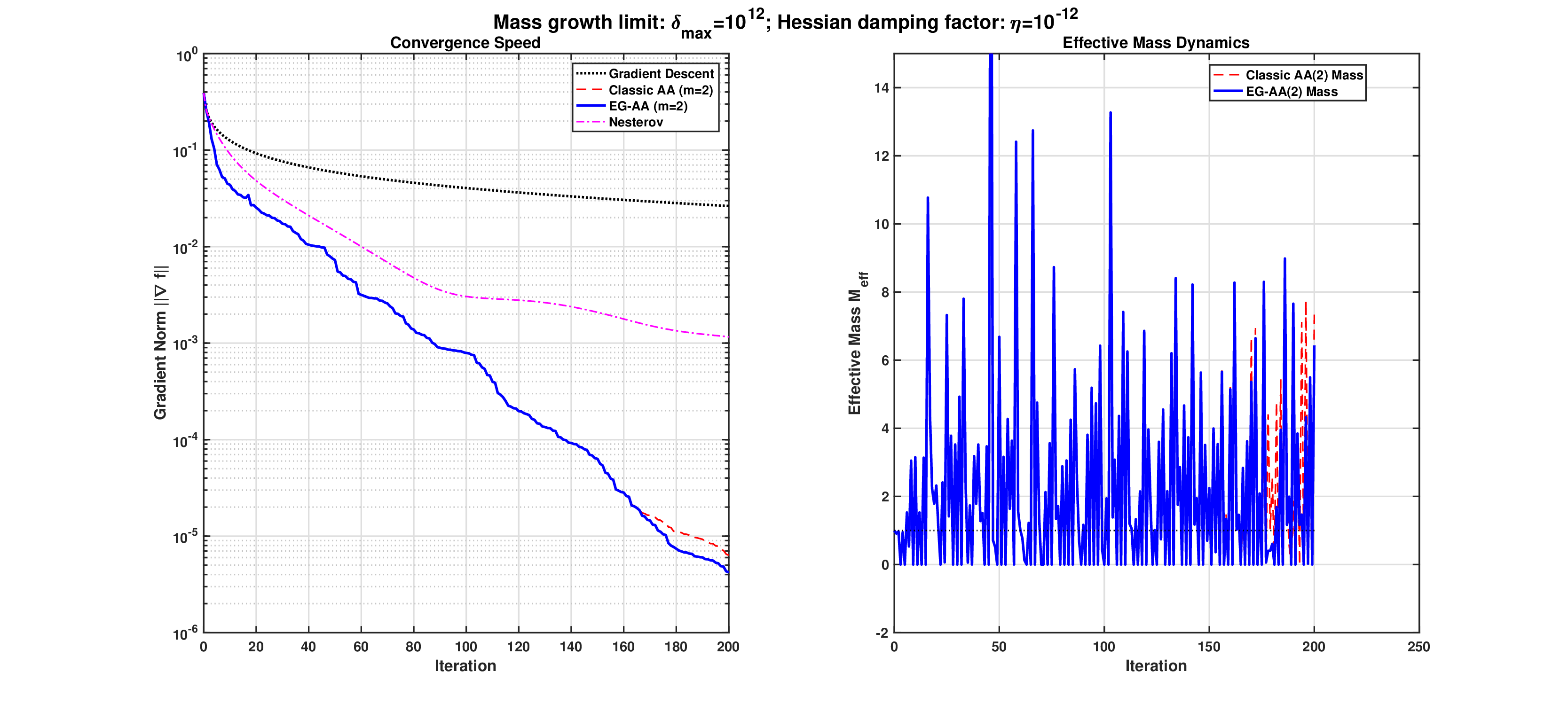}}
    \makebox[\textwidth][c]{\includegraphics[width=0.82\textwidth]{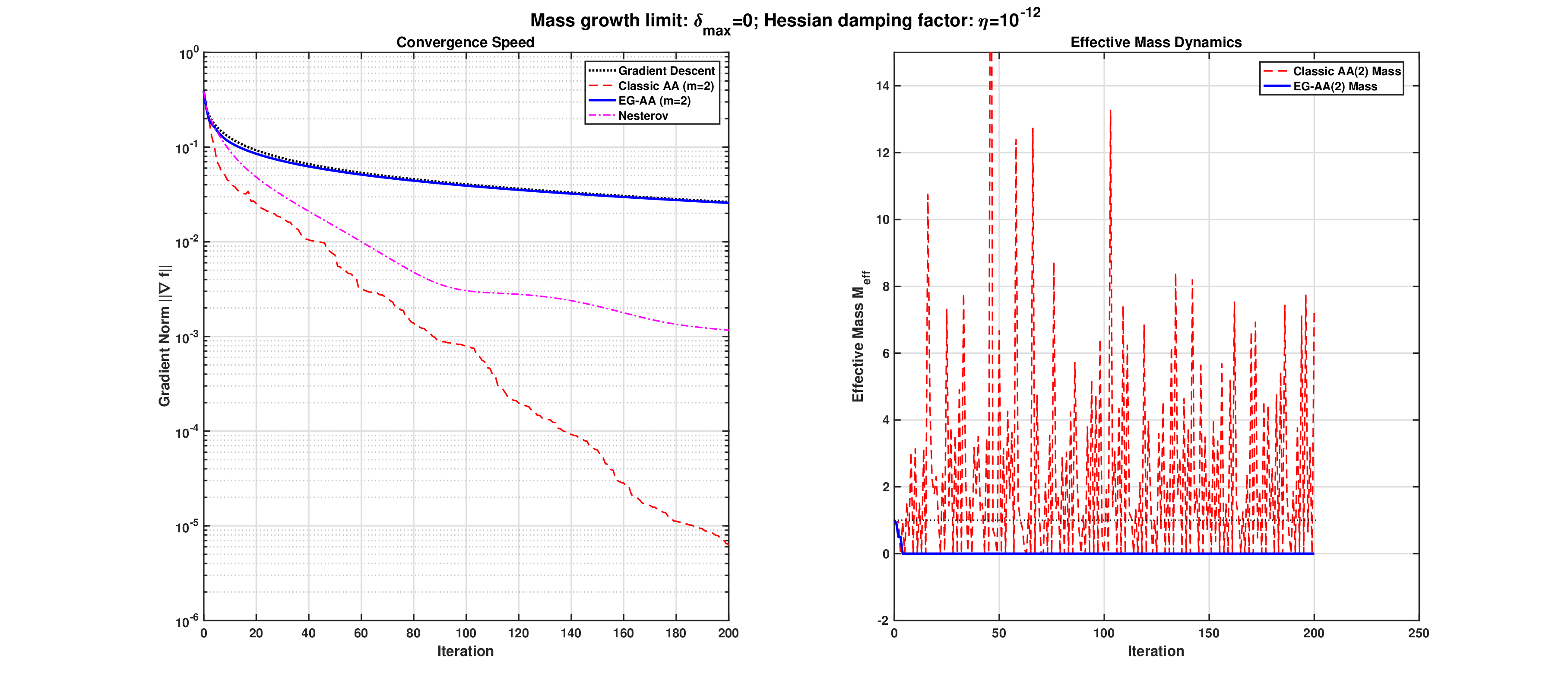}}
    \caption{\textbf{(Left)} Convergence of gradient norm. \textbf{(Right)} Effective Mass dynamics. When constraints are relaxed ($\delta_{max} \to \infty, \eta \to 0$), EG-AA (Blue) recovers the behavior of Standard AA (Red); when when $\delta_{max}$ approaches zero, EG-AA recovers GD method.}
    \label{fig:Logistic}
\end{figure}


\subsection{Example 3: Ill-Conditioned Non-negative Least Squares}

To rigorously evaluate robustness in constrained settings, we construct a hard instance of the Non-negative Least Squares (NNLS) problem. This setup is specifically designed to be ill-conditioned and sparse, challenging standard acceleration techniques that rely on smooth historical gradients.

\paragraph{Problem Formulation}
The objective function incorporates an $L_2$-regularized least squares loss with non-negativity constraints:
\begin{equation}
    \min_{x \in \mathbb{R}^n} f(x) = \frac{1}{2M} \|Ax - b\|_2^2 + \mu \|x\|_2^2, \quad \text{s.t.} \quad x \ge 0.
\end{equation}

\paragraph{Data Generation Protocol}
The problem parameters are synthesized to simulate a high-dimensional, noisy, and numerically stiff environment:
\begin{itemize}
    \item \textbf{Dimensions:} We set the number of samples $M = 2000$ and the feature dimension $n = 500$.
    \item \textbf{Ill-Conditioned Design Matrix:} The matrix $A$ is constructed via Singular Value Decomposition (SVD) $A = U \Sigma V^\top$. To induce stiffness, the singular values are logarithmically spaced in the interval $[10^{-4}, 1]$ and scaled by $\sqrt{M}$. This scaling ensures that the data term $\frac{1}{M}A^\top A$ has a unit spectral norm while maintaining a high condition number $\kappa \approx 10^4$.
    \item \textbf{Sparse Ground Truth:} The true solution $x_{true}$ is generated with extreme sparsity ($p=0.5\%$), resulting in approximately 2-3 active non-zero elements drawn uniformly from $[1, 2]$. This sparsity forces the optimization trajectory to frequently interact with the non-smooth boundaries ($x_i = 0$).
    \item \textbf{Observation Noise:} The target vector is generated as $b = A x_{true} + \epsilon$, with Gaussian noise $\epsilon \sim \mathcal{N}(0, 0.5^2)$.
    \item \textbf{Weak Regularization:} We choose small regularization parameters $\mu = 5 \times 10^{-4}, \mu = 5 \times 10^{-5} and \mu = 5 \times 10^{-6}$. Unlike larger values (e.g., $\mu=0.1$) that would artificially improve the condition number of the Hessian ($H \approx \frac{1}{M}A^\top A + 2\mu I$), this small $\mu$ preserves the intrinsic ill-conditioning of the data term.
\end{itemize}

\paragraph{Algorithm Configuration}
The Lipschitz constant of the gradient is estimated as $L \approx \lambda_{\max}(\frac{1}{M}A^\top A) + 2\mu \approx 1$. All algorithms use a conservative step size $\beta = 1/L$. The window size for both AA and EG-AA is fixed as $m=3$. We use a fixed mass growth limit $\delta_{max}=20$ and fixed Hessian damping factor $\eta=0.02$ in EG-AA for all cases. All solvers are initialized at the origin $x_0 = \mathbf{0}$ and executed for a maximum of 5000 iterations to observe long-term stability. 

\paragraph{Results}

As shown in Figure \ref{fig:NNLS}, for this hard ill-conditioned non-negative least squares problem, EG-AA also outperforms Standard AA and Nesterov method.

\begin{figure}[htbp!]
    \centering
    \makebox[\textwidth][c]{\includegraphics[width=1.0\textwidth]{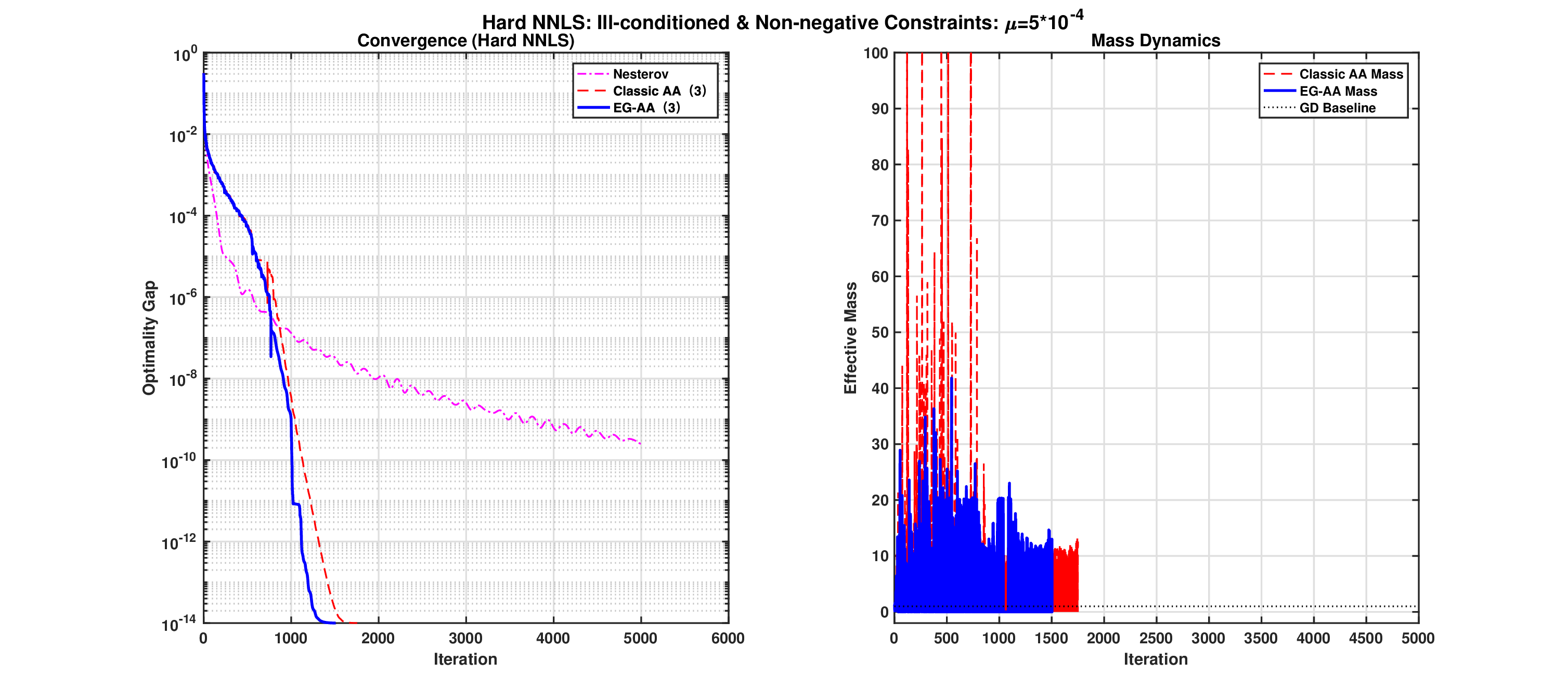}} 
    \makebox[\textwidth][c]{\includegraphics[width=1.0\textwidth]{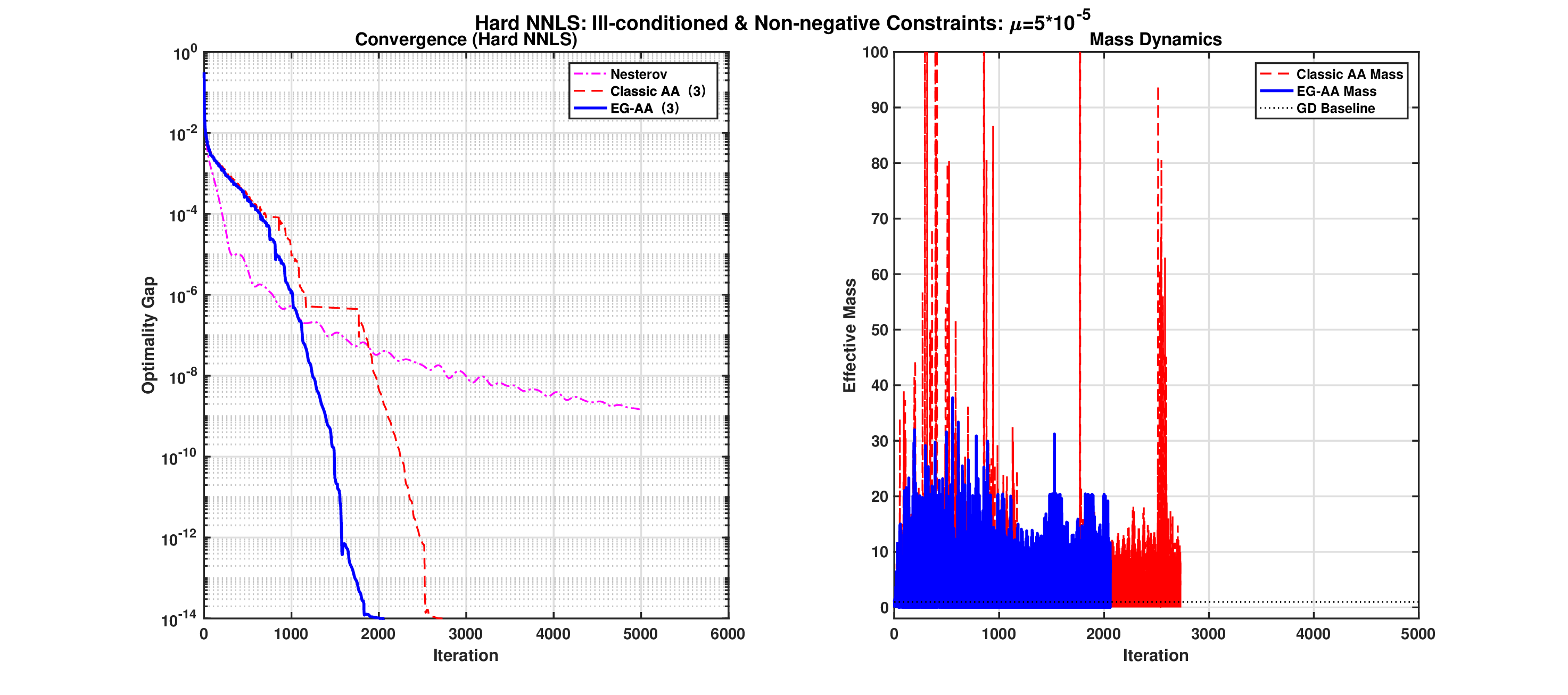}} 
    \makebox[\textwidth][c]{\includegraphics[width=1.0\textwidth]{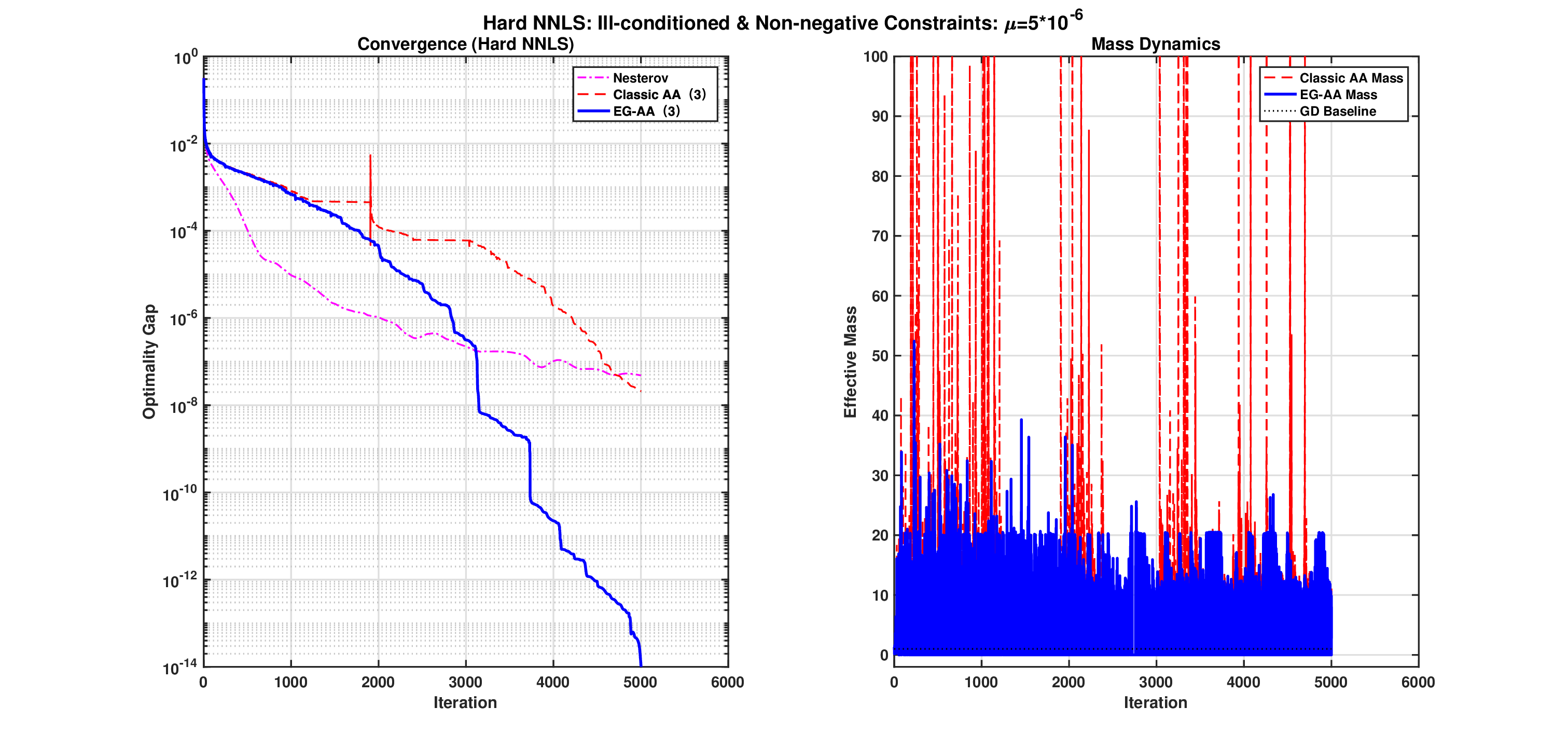}} 
    \caption{Comparison of optimization algorithms on an ill-conditioned non-negative least squares problem. \textbf{(Left)} Convergence results showing EG-AA's superior rate. \textbf{(Right)} Effective mass evolution. As the regularization parameter $\mu$ becomes smaller, EG-AA gradually outperforms Standard AA.}
    \label{fig:NNLS}
\end{figure}


\subsection{Example 4: Ill-Conditioned Quadratic Problem}

Finally, we consider the unconstrained minimization of a quadratic objective function:
\begin{equation}
    f(x) = \frac{1}{2} x^\T A x, \quad x \in \R^n.
\end{equation}
The problem instance is synthesized to simulate a \textbf{stiff} optimization landscape, characterized by the following specifications:
\begin{itemize}
    \item \textbf{Dimension:} $n = 100$.
    \item \textbf{Condition Number:} $\kappa = \frac{\lambda_{\max}(A)}{\lambda_{\min}(A)} \in \{50, 500, 5000, 50000\}$. This high condition number implies the existence of extremely narrow valleys, a scenario where standard momentum methods often exhibit oscillatory behavior.
    \item \textbf{Spectrum:} The matrix $A$ is diagonal with eigenvalues uniformly spaced in the interval $[1, \kappa]$.
    \item \textbf{Lipschitz Constant:} $L = \lambda_{\max} = \kappa$.
\end{itemize}

\paragraph{Algorithm Configuration}
We benchmark our proposed EG-AA against three baseline algorithms: Gradient Descent (GD), Nesterov Accelerated Gradient (NAG), and Standard Anderson Acceleration (AA).
\begin{enumerate}
    \item \textbf{Baselines (GD \& NAG):} We utilize the standard conservative step size $\beta_{GD} = \beta_{NAG} = 1.0/L$ to guarantee stability.
    \item \textbf{Anderson Methods (AA \& EG-AA):} To rigorously test robustness, we employ an \textbf{aggressive step size} $\beta_{AA} = 1.8/L$. It is worth noting that for standard gradient descent, a step size approaching $2/L$ represents the theoretical stability limit and typically results in oscillatory divergence in practice.
    \item \textbf{EG-AA Specifics:} Unless otherwise stated, we set the memory depth $m=3$, the mass growth limit $\delta_{max} = 5.0$, and the Hessian damping factor $\eta = 0.1$.
    \item \textbf{Initialization:} The initialization $x_0$ is drawn from a standard normal distribution $\mathcal{N}(0, I_n)$.
\end{enumerate}

\paragraph{Convergence Analysis and Discussion}

The convergence trajectories, effective mass dynamics, and acceleration gain factors are visualized in Figure \ref{fig:quadratic_comparison}. As illustrated in Figure \ref{fig:quadratic_comparison} (Left), standard Gradient Descent (GD) suffers from severe stagnation due to the ill-conditioned spectrum. \textbf{EG-AA} consistently outperforms standard AA and Nesterov acceleration as the condition number $\kappa$ increases from 50 to 50,000. 

Moreover, as shown in Figure \ref{fig:quadratic_comparison} (Center), EG-AA effectively manages the system's inertia. When standard AA is stable, EG-AA allows for mass accumulation (energy injection) to further speed up convergence. However, unlike standard AA, it prevents unbounded mass growth that leads to instability. Lastly, the sensitivity analysis regarding the Hessian damping factor $\eta$ and the memory depth $m$ is presented in Figure \ref{fig:quadratic_comparison_parameter}, demonstrating the algorithm's robustness to parameter variations.

\begin{figure}[htbp!]
    \centering
    \makebox[\textwidth][c]{\includegraphics[width=1.0\textwidth]{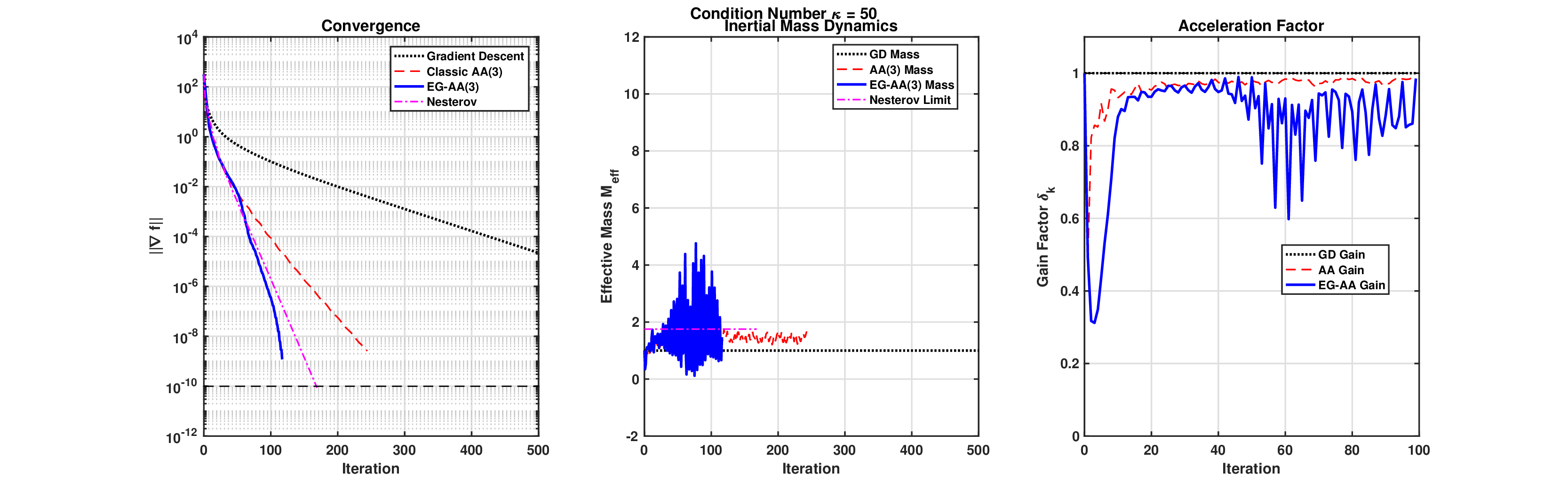}} 
    \vspace{0.2cm}
    \makebox[\textwidth][c]{\includegraphics[width=1.0\textwidth]{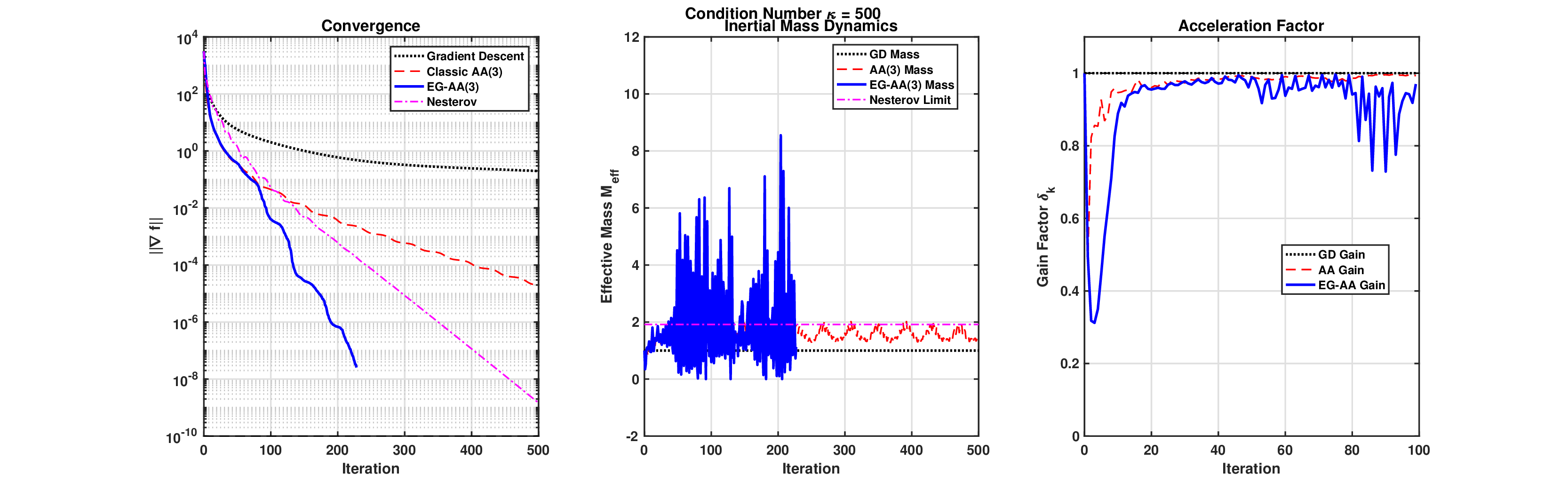}} 
    \vspace{0.2cm}
    \makebox[\textwidth][c]{\includegraphics[width=1.0\textwidth]{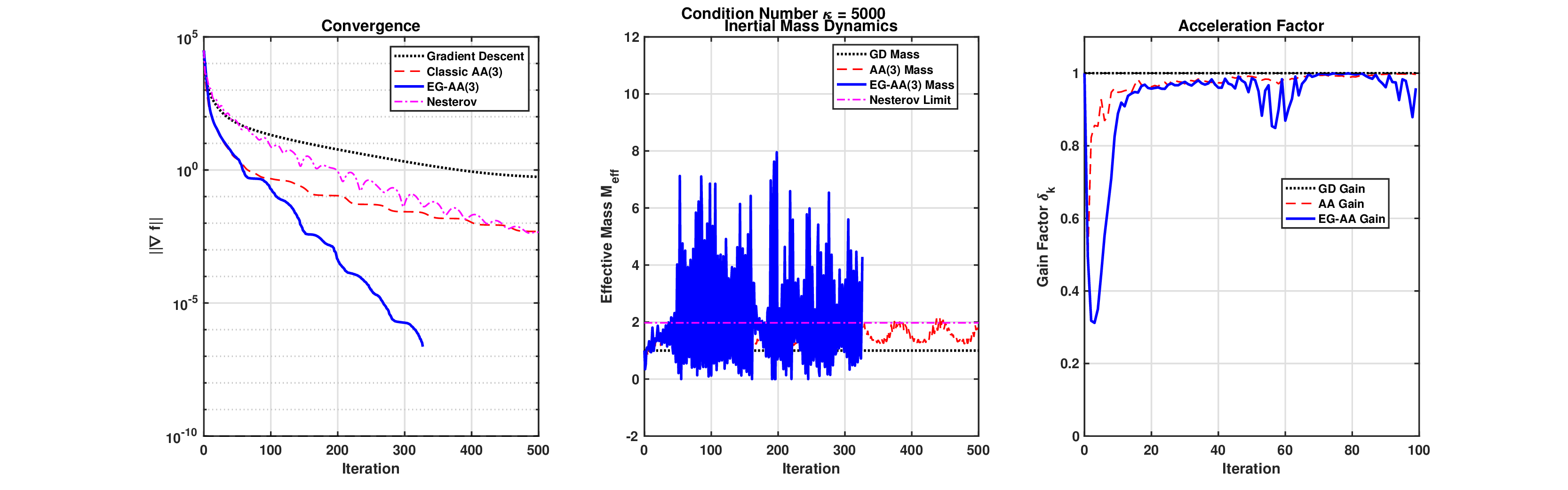}} 
    \vspace{0.2cm}
    \makebox[\textwidth][c]{\includegraphics[width=1.0\textwidth]{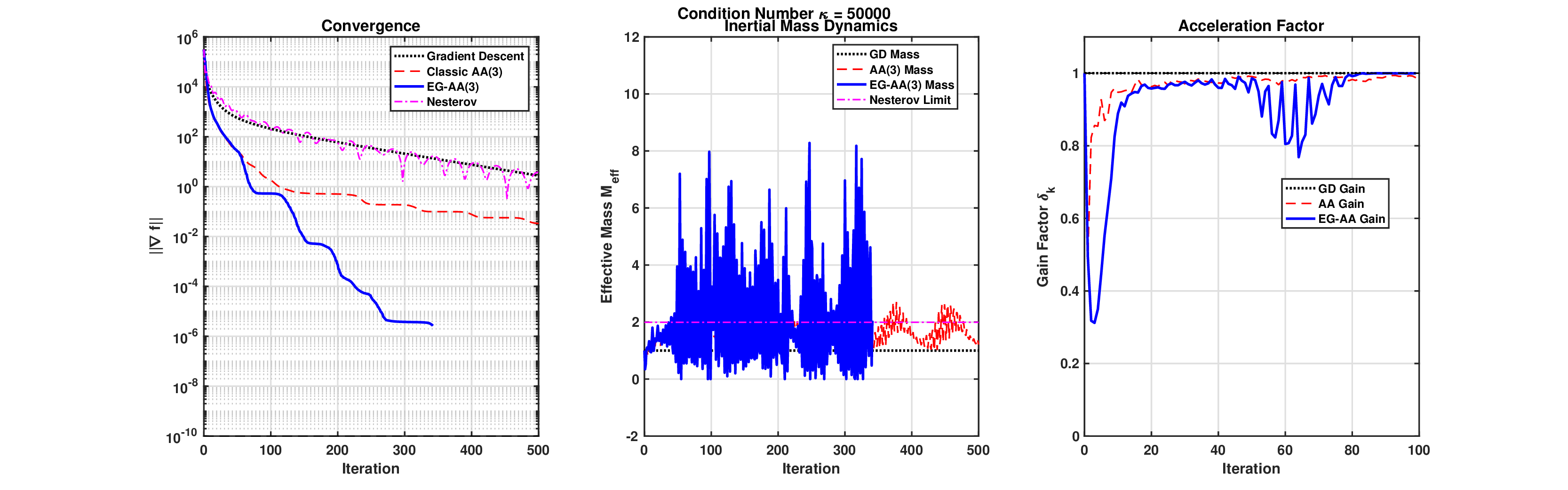}} 
    \caption{Comparison of optimization algorithms on an ill-conditioned quadratic problem with varying condition numbers ($\kappa \in \{50, 500, 5000, 50,000\}$). \textbf{(Left)} Convergence of gradient norm. EG-AA (Blue) achieves faster convergence than Nesterov (Pink) and is significantly more stable than Standard AA (Red). \textbf{(Center)} Effective Mass dynamics. EG-AA intelligently modulates mass to accelerate convergence when stable. \textbf{(Right)} Acceleration Gain Factor $\delta_k$. Lower values indicate stronger geometric contraction.}
    \label{fig:quadratic_comparison}
\end{figure}

\begin{figure}[htbp!]
    \centering
    \makebox[\textwidth][c]{\includegraphics[width=1.0\textwidth]{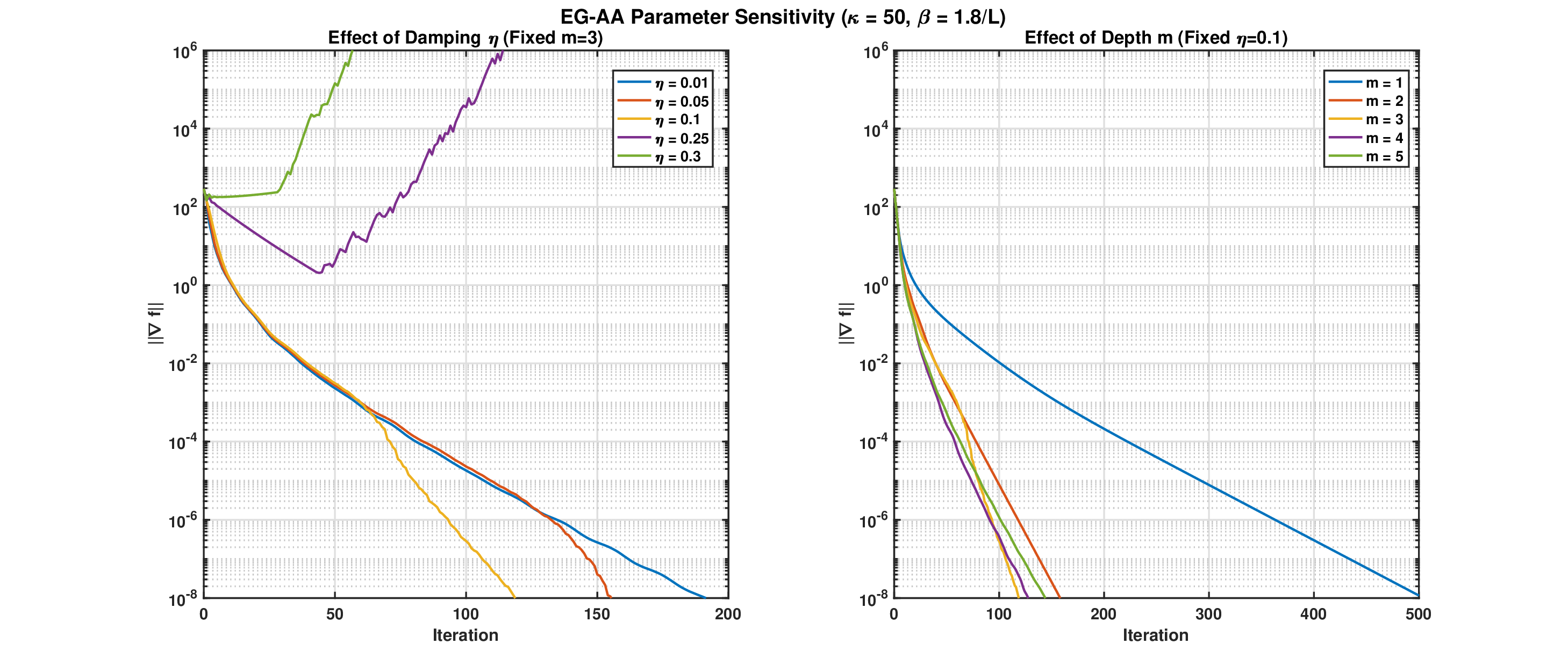}} 
    \vspace{0.2cm}
    \makebox[\textwidth][c]{\includegraphics[width=1.0\textwidth]{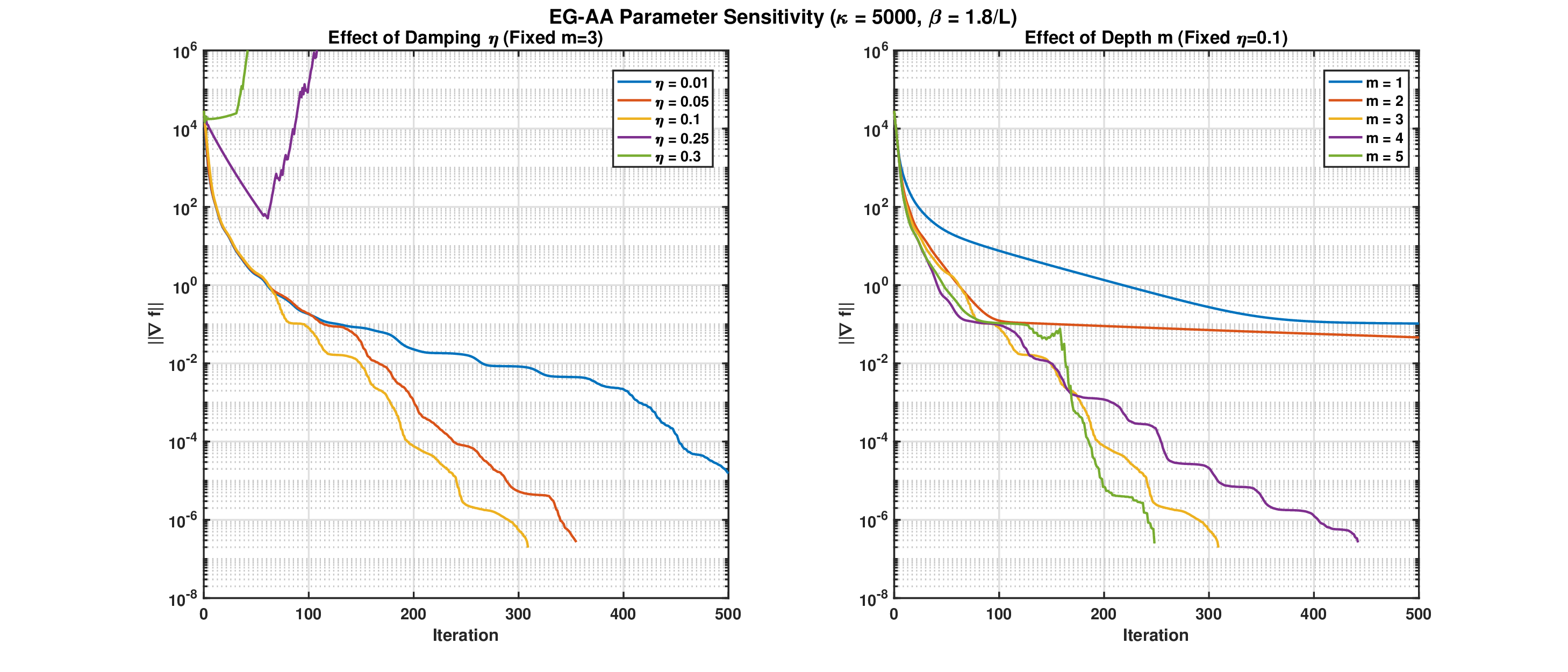}} 
    \caption{Parameter sensitivity analysis on an ill-conditioned quadratic problem ($\kappa \in \{50, 5000\}$). \textbf{(Left)} The effect of the Hessian damping factor $\eta$. \textbf{(Right)} The effect of the memory depth $m$.}
    \label{fig:quadratic_comparison_parameter}
\end{figure}

\section{Conclusion}

In this paper, we have bridged the gap between discrete Anderson Acceleration (AA) and continuous-time inertial dynamics, offering a unified theoretical framework for understanding and enhancing acceleration methods. By rigorously transforming AA into an adaptive multi-step momentum scheme as in \cite{Attouch2020, Xie2025}, we revealed that its continuous limit corresponds to a second-order ODE with \textbf{variable effective mass}! This physical perspective uncovers the fundamental mechanism of instability in standard AA: the unchecked growth of effective mass acts as a form of ``negative damping,'' injecting energy into the system and violating passivity constraints in stiff regimes.

Conversely, our high-resolution analysis reveals an implicit Hessian-driven damping term inherent to AA, which provides critical stabilization in high-curvature directions~\cite{Shi2018}. Building on these insights, we introduce the \textbf{Energy-Guarded Anderson Acceleration (EG-AA)}. Functioning as an inertial governor that explicitly bounds mass growth and enforces energy dissipation, EG-AA effectively resolves the stability-efficiency trade-off. Our convergence analysis, based on the Acceleration Gain Factor, proves that EG-AA maintains the rapid geometric contraction of standard AA in well-conditioned subspaces while suppressing nonlinear error amplification. Numerical experiments on ill-conditioned convex composite problems confirm that EG-AA achieves superior stability and convergence rates compared to standard AA and Nesterov acceleration, validating the practical utility of our thermodynamic framework. Future work will extend this energy-based perspective to non-smooth and non-convex regimes, and specifically to the ADMM algorithm.

\section*{Data Availability Statement}

All data and code supporting the findings of this study are available at the GitHub repository: 

\texttt{https://github.com/kewangchen/EG-AA}. 

The repository will be made publicly available upon manuscript acceptance. During the review process, it can be accessed by request from the corresponding author.

\section*{Conflicts of Interest:}
The authors declare that they have no known competing financial interests or personal relationships that could have appeared to influence the work reported in this paper.
\newpage
\bibliographystyle{siamplain}

\appendix
\section{Derivation of the ODE Limit from the Discrete Nesterov Scheme}

We start with the discrete update equation for Nesterov's Accelerated Gradient method after eliminating the auxiliary variable $y^k$:
\begin{equation} \label{eq:discrete_nesterov}
    x^{k+1} - x^k = \frac{k-1}{k+2}(x^k - x^{k-1}) - h \nabla f(y^k).
\end{equation}

\noindent \textbf{Step 1: Continuous Time Ansatz} \\
Let $h$ be the step size. We introduce the continuous time variable $t = k\sqrt{h}$. Thus, the discrete iteration count relates to time as $k = t/\sqrt{h}$. We assume there exists a smooth trajectory $x(t)$ such that $x^k \approx x(t)$ as $h \to 0$. Specifically:
\begin{align*}
    x^k     &= x(t), \\
    x^{k+1} &= x(t + \sqrt{h}), \\
    x^{k-1} &= x(t - \sqrt{h}).
\end{align*}

\noindent \textbf{Step 2: Taylor Expansions} \\
We apply the Taylor series expansion to $x(t \pm \sqrt{h})$ around $t$ up to the second order, treating $\sqrt{h}$ as the small perturbation:
\begin{align}
    x^{k+1} &= x(t) + \sqrt{h}\dot{x}(t) + \frac{h}{2}\ddot{x}(t) + \mathcal{O}(h^{3/2}), \label{eq:taylor_forward} \\
    x^{k-1} &= x(t) - \sqrt{h}\dot{x}(t) + \frac{h}{2}\ddot{x}(t) + \mathcal{O}(h^{3/2}). \label{eq:taylor_backward}
\end{align}
Using \eqref{eq:taylor_forward} and \eqref{eq:taylor_backward}, we can approximate the finite difference terms in \eqref{eq:discrete_nesterov}:
\begin{align}
    \text{LHS: } \quad x^{k+1} - x^k &= \sqrt{h}\dot{x}(t) + \frac{h}{2}\ddot{x}(t) + \mathcal{O}(h^{3/2}), \label{eq:LHS} \\
    \text{Momentum: } \quad x^k - x^{k-1} &= \sqrt{h}\dot{x}(t) - \frac{h}{2}\ddot{x}(t) + \mathcal{O}(h^{3/2}). \label{eq:momentum_term}
\end{align}

\noindent \textbf{Step 3: Asymptotic Expansion of the Momentum Coefficient} \\
The coefficient $\frac{k-1}{k+2}$ must be expressed in terms of $t$ and $h$. Using $k = t/\sqrt{h}$, we have:
\begin{equation} \label{eq:coeff_expansion}
    \frac{k-1}{k+2} = \frac{k+2-3}{k+2} = 1 - \frac{3}{k+2} \approx 1 - \frac{3}{k} = 1 - \frac{3\sqrt{h}}{t}.
\end{equation}
Note that higher-order terms of $1/k$ correspond to powers of $\sqrt{h}$ greater than 1, which become negligible in the limit.

\noindent \textbf{Step 4: Gradient Term Approximation} \\
Since $y^k = x^k + \beta_k(x^k - x^{k-1}) \approx x(t) + \mathcal{O}(\sqrt{h})$, by continuity of $\nabla f$, we have $\nabla f(y^k) = \nabla f(x(t)) + \mathcal{O}(\sqrt{h})$.
Multiplying by $h$ in the update rule makes the error term $\mathcal{O}(h^{3/2})$, which is negligible. Thus:
\begin{equation} \label{eq:grad_approx}
    h \nabla f(y^k) \approx h \nabla f(x(t)).
\end{equation}

\noindent \textbf{Step 5: Substituting and Taking the Limit} \\
Substituting \eqref{eq:LHS}, \eqref{eq:momentum_term}, \eqref{eq:coeff_expansion}, and \eqref{eq:grad_approx} into the original discrete equation \eqref{eq:discrete_nesterov}:

\begin{equation*}
    \left( \sqrt{h}\dot{x} + \frac{h}{2}\ddot{x} \right) = \left( 1 - \frac{3\sqrt{h}}{t} \right) \left( \sqrt{h}\dot{x} - \frac{h}{2}\ddot{x} \right) - h \nabla f(x).
\end{equation*}

\noindent Expanding the right-hand side (ignoring terms of order $\mathcal{O}(h^{3/2})$ and higher):
\begin{align*}
    \sqrt{h}\dot{x} + \frac{h}{2}\ddot{x} &= \left( \sqrt{h}\dot{x} - \frac{h}{2}\ddot{x} \right) - \frac{3\sqrt{h}}{t}\left( \sqrt{h}\dot{x} \right) - h \nabla f(x) \\
    \sqrt{h}\dot{x} + \frac{h}{2}\ddot{x} &= \sqrt{h}\dot{x} - \frac{h}{2}\ddot{x} - \frac{3h}{t}\dot{x} - h \nabla f(x).
\end{align*}

\noindent Subtracting $\sqrt{h}\dot{x}$ from both sides and collecting the $\ddot{x}$ terms:
\begin{equation*}
    \frac{h}{2}\ddot{x} + \frac{h}{2}\ddot{x} = - \frac{3h}{t}\dot{x} - h \nabla f(x).
\end{equation*}

\noindent Simplifying:
\begin{equation*}
    h \ddot{x}(t) + \frac{3h}{t}\dot{x}(t) + h \nabla f(x(t)) = 0.
\end{equation*}

\noindent Finally, dividing by $h$ (assuming $h > 0$), we obtain the second-order ODE:
\begin{equation} \label{eq:final_ode}
    \ddot{x}(t) + \frac{3}{t}\dot{x}(t) + \nabla f(x(t)) = 0.
\end{equation}
This completes the derivation.

\end{document}